\documentclass[12pt]{amsart}

\usepackage{amsmath}
\usepackage{amssymb}
\usepackage{array}
\usepackage{enumitem}
\usepackage{verbatim} 
\usepackage{hyperref}

\input cyracc.def
\newfam\cyrfam

\theoremstyle{plain}
\newtheorem{theorem}{Theorem}[section]
\newtheorem{lemma}[theorem]{Lemma}

\newtheorem{corollary}[theorem]{Corollary}

\newtheorem{definition}[theorem]{Definition}
\theoremstyle{remark}
\newtheorem*{remark}{Remark}

\newtheorem*{notation}{Notation}

\numberwithin{equation}{section}

\newcommand{\bP}{\mathbb{P}}
\newcommand{\bR}{\mathbb{R}}

\newcommand{\Z}{\mathbb{Z}}

\newcommand{\cX}{\mathcal{X}}
\newcommand{\cY}{\mathcal{Y}}
\newcommand{\cP}{\mathcal{P}}

\newcommand{\ga}{\mathfrak{a}}

\newcommand{\bV}{\mathbf{V}}
\newcommand{\bT}{\mathbf{T}}

\newcommand{\Hom}{\mathrm{Hom}}
\newcommand{\Aut}{\mathrm{Aut}}

\newcommand{\Bij}{\mathrm{Bij}}
\newcommand{\Map}{\mathrm{Map}}
\newcommand{\Gras}{\mathit{Gras}}
\newcommand{\Rel}{\mathrm{Rel}}

\newcommand{\id}{\mathrm{id}}

\newcommand{\Gamm}{\mathbf{\Gamma}} 
\newcommand{\Beta}{\mathrm{B}}
\newcommand{\Kappa}{\mathrm{K}}

\newcommand{\setof}[2]{\{ #1 \mid #2\}}
\newcommand{\Bigsetof}[2]{\begin{Bmatrix} #1 \,\Big|\, #2 \end{Bmatrix}}

\newcommand{\inv}{^{-1}}

\newcommand{\msk}{\medskip}
\newcommand{\ssk}{\smallskip}
\newcommand{\nin}{\noindent}

\begin{document}

\title{The projective geometry of a group}

\author{Wolfgang Bertram}

\address{Institut \'{E}lie Cartan Nancy \\
Lorraine Universit\'{e}, CNRS, INRIA \\
Boulevard des Aiguillettes, B.P. 239 \\
F-54506 Vand\oe{}uvre-l\`{e}s-Nancy, France}

\email{\url{bertram@iecn.u-nancy.fr}}



\begin{abstract}
We show that the pair $(\cP(\Omega),\Gras(\Omega))$ 
given by the power set $\cP = \cP(\Omega)$ 
and by the ``Grassmannian'' $\Gras(\Omega)$ of all subgroups of an arbitrary group $\Omega$
behaves very much like a projective space $\bP(W)$ and its dual projective space
$\bP(W^*)$ of a vector space $W$.
More precisely, we generalize several results from the case of  the abelian group
$\Omega =(W,+)$ (cf.\ \cite{BeKi10}) to the case of a general group $\Omega$. Most notably, 
pairs of subgroups $(a,b)$ of $\Omega$ parametrize {\em torsor} and
{\em semitorsor} structures on $\cP$. The r\^ole of associative algebras and -pairs
from \cite{BeKi10} is now taken by analogs of {\em near-rings}.
\end{abstract}

\maketitle

\section{Introduction and statement of main results}

\subsection{Projective geometry of an abelian group}
Before explaining our general results, let us briefly recall the classical case of projective
geometry of a vector space $W$ : let $\cX = \bP( W)$ be the projective space of $W$ and
$\cX' = \bP (W^*)$ be its
 dual projective space (space of hyperplanes). 
The ``duality'' between $\cX$ and $\cX'$ is encoded on two levels 
\begin{enumerate}
\item[(1)]  
on the level of {\em incidence structures}: an element $x=[v] \in \bP W$ is {\em 
incident} with an element $a=[\alpha] \in \bP W^*$ if ``$x$ lies on $a$'', i.e., if
$\alpha(v)=0$ ; otherwise we say
that they are {\em remote} or {\em transversal}, and we  then write $x\top a$ ;
\item[(2)]  
on the level of  {\em (linear) algebra}:
the set $a^\top$ of elements $x \in \cX$ that are transversal to $a$ is,
in a completely natural way, an {\em affine space}.
\end{enumerate}
In \cite{BeKi10}, the second point has been generalized: for any
pair $(a,b) \in \cX' \times \cX'$, the intersection 
$U_{ab}:=a^\top \cap b^\top$ of two ``affine cells'' carries a natural torsor structure.
Recall that ``torsors are for groups what affine spaces are for vector 
spaces'':\footnote{The concept used here goes back to J.\ Certaine \cite{Cer43}; there
are several equivalent versions,  known
under various other names such as {\em groud}, {\em heap}, 
or {\em principal homogeneous space}.}

\begin{definition}\label{TorsorDefinition}
A {\em semitorsor} is a set $G$ together with a map
$G^3 \to G$, $(x,y,z) \mapsto (xyz)$ such that the following identity, 
called the \emph{para-associative law}, holds:
\[
(xy(zuv))= (x(uzy)v)=((xyz)uv)\, . \tag{T1}
\]
A {\em torsor} is a semitorsor in which, moreover, the following
\emph{idempotent law} holds:
\[
(xxy)=y=(yxx) \, . \tag{T2}
\]
\end{definition}

\nin Fixing the middle element $y$ in a torsor $G$, we get a group law
$xz:=(xyz)$ with neutral element $y$, and every group is obtained in this way.
Similarly, semitorsors give rise to semigroups, but the converse is more
complicated. 
The torsors $U_a:=U_{aa}$ are the underlying torsors of the affine space $a^\top$, hence
are abelian, whereas for $a \not= b$, the torsors $U_{ab}$ are in general 
 non-commutative.
Thus, in a sense, the torsors $U_{ab}$ are {\em deformations} of the abelian torsor $U_a$.
More generally, in \cite{BeKi10} all this is done
 for a pair $(\cX,\cX')$ of {\em dual Grassmannians}, not only for
projective spaces.

\subsection{Projective geometry of a general group}
In the present work, the {\em commutative} group $(W,+)$ will be replaced by an
arbitrary group $\Omega$ (however, in order to keep formulas easily readable, we will
still use an additive notation for the group law of $\Omega$). 
%
It  turns out, then,  that the r\^ole of $\cX$ is taken by the power set $\cP(\Omega)$ of
{\em all} subsets of $\Omega$, and the one of $\cX'$ by the
``Grassmannian'' of all subgroups of $\Omega$.
We call $\Omega$  the ``background group'', or just {\em the background}. Its
subsets  will be denoted by small latin letters $a,b,x,y,\ldots$ and, if possible, elements of such sets
by corresponding greek letters: $\alpha \in a$, $\xi \in x$, and so on.
As said above, ``projective geometry on $\Omega$'' in our sense has two 
ingredients which we are going to explain now
\begin{enumerate}
\item[(1)]  a (fairly weak) {\em incidence} (or rather: {\em non-incidence}) {\em structure},
and 
\item[(2)]   a much more relevant {\em algebraic} structure consisting of a collection of
torsors and  semi-torsors.
\end{enumerate}

\begin{definition}\label{TransversalDefinition}
The {\em projective geometry} of a group $(\Omega,+)$ is its power set
$\cP:=\cP(\Omega)$. 
We say that a pair $(x,y) \in \cP^2$ is {\em left transversal} if
 every $\omega \in \Omega$ admits a unique decomposition
$\omega=\xi + \eta$ with $\xi \in x$ and $\eta \in y$.
We then write $x \top y$.
 We say that the pair $(x,y)$ is {\em right transversal} if $y \top x$, and we
 let
$$
x^\top := \{ y \in \cP \mid x \top y \}, \qquad { }^\top x:= \{ y \in \cP \mid y \top x \} \, .
$$
\end{definition}

\nin The ``(non-) incidence structure'' thus defined is  not very interesting in its own
right;  however, in combination with the algebraic torsor structures it becomes quite powerful.
 There are two, in a certain sense ``pure'', special
 cases to consider;   the general case is a sort of mixture of these two:
 let $a,b$ be two subgroups of $\Omega$,

\ssk
(A)
the {\em transversal case}
$a \top b$:  then $a^\top \cap { }^\top b$ is a torsor of ``bijection type'',

(B) the {\em singular case}  $a=b$; it corresponds to  ``pointwise torsors''
${ }^\top a$ and $b^\top$.

\ssk
\nin 
Protoypes for (A) are torsors of the type $G=\Bij(X,Y)$ (set of bijections
$f:X \to Y$ between two sets $X$ and $Y$), with torsor structure $(fgh):=f \circ g\inv \circ h$,
and prototypes for (B) are torsors of the type $G=\Map(X,A)$ 
(set of maps from $X$ to $A$), where $A$ is a torsor and $X$ a set, 
together  with their natural  ``pointwise product''. 

\ssk
Case (A) arises, if, when $a \top b$,  we identify $\Omega$ with the cartesian product
$a \times b$; then elements $z \in { }^\top b$ can be identified with
``left graphs'' $\{ (\alpha,Z \alpha) \mid \, \alpha \in a \}$ of maps $Z:a \to b$.
The map $Z$ is bijective iff this graph belongs to $a^\top$.
Therefore $G:=  { }^\top b \cap a^\top$ carries a natural torsor structure of ``bijection type'': 
as a torsor,
it is isomorphic to $\Bij(a,b)$. It may be empty; if it is non-empty, then it is isomorphic
to (the underlying torsor of) the group $\Bij(a,a)$. Note that the structure of
this group does not involve the one of $\Omega$, indeed, the group structure
of $\Omega$ enters here only implicitly, via the identification of $\Omega$ with
$a \times b$. 

\ssk
On the other hand, in the ``singular case'' (B),  the set
${ }^\top a$ is naturally identified with the set of sections of the canonical
projection $\Omega \to \Omega/a$, and this set is a torsor of pointwise type,
modelled on the ``pointwise group'' of all maps $f:\Omega/a \to a$. 
It is abelian iff so is $a$. Indeed, such torsors correspond precisely to the
``affine cells'' from usual projective geometry. 

\subsection{The ``balanced'' torsors $U_{ab}$ and the ``unbalanced'' torsors $U_a$.}
Following the ideas developed in \cite{BeKi10}, we consider
 the general torsors $U_{ab}$  as a sort of ``deformation of the pure case (B)
in direction of (A)''. However, for treating the case of a non-commutative group $\Omega$
we need several
important modifications of the setting from \cite{BeKi10}:
 first of all, the projective geometry $\cP$ and its ``dual''
$\Gras(\Omega)$ are no longer the same objects (the subset $\Gras(\Omega) \subset \cP$ is no longer
stable under the various torsor laws); next, for $a=b$, we have to distinguish
between several versions of torsor laws, those that one can deform easily, called
{\em balanced}, and those which seem to be more rigid and which we
call {\em unbalanced}. The most conceptual way to present this is via
 the following algebraic {\em structure maps}:

\begin{definition}
The {\em structure maps} of  $(\Omega,+)$ are the
maps $\Gamma:\cP^5 \to \cP$ and
$\Sigma:\cP^4 \to \cP$ 
defined, for
 $x,a,y,b,z \in \cP$, by
\begin{align*}
\Gamma(x,a,y,b,z) &:=
\Bigsetof{\omega \in \Omega}
{\begin{array}{c}
\exists \xi \in x,
\exists \alpha \in a,
\exists \eta \in y,
\exists \beta \in b,
\exists \zeta \in z : \\
\xi = \omega  + \beta,\quad
\eta =  \alpha +  \omega + \beta, \quad
\zeta = \alpha + \omega
\end{array}} ,
\cr
\Sigma(a,x,y,z) &:=
\Bigsetof{\omega \in \Omega}
{\begin{array}{c}
 \exists \xi \in x,\exists \eta \in y,\exists \zeta \in z,\exists \beta,\beta' \in b: \\
\xi = \omega + \beta, \quad
\eta = \omega + \beta' + \beta, \quad
\zeta = \omega +\beta' 
\end{array}} \, .
\end{align*}
For a fixed pair $(a,b) \in \cP^2$, resp.\ a fixed element $b \in \cP$, we let
$$
(xyz)_{ab}:= \Gamma(x,a,y,b,z), \qquad
(xyz)_b:=\Sigma(b,x,y,z) \ .
$$
\end{definition}

\nin The following is a main result of the present work (Theorems  \ref{PointwisetorsorTheorem}
and \ref{BalancedTorsorTheorem}):

\begin{theorem}\label{TheoremA}
Assume $(a,b)$ is a pair of subgroups of $\Omega$. 
Then the following holds:
\begin{enumerate}
\item
The map $(x,y,z) \mapsto (xyz)_{b}$ defines a torsor structure on 
the set $ { }^\top b$.
We denote this torsor by $U_b$. It 
 is isomorphic to the torsor of sections of the projection $\Omega \to  \Omega / b$.
\item
The map $(x,y,z) \mapsto (xyz)_{ab}$ defines a torsor structure on 
the set $ a^\top \cap { }^\top b$. 
We denote this torsor by $U_{ab}$. 
If, moreover, $a \top b$, then it is isomorphic to the group of bijections of $a$.
\end{enumerate}
\end{theorem}

\nin Just as  the torsor structures considered in \cite{BeKi10}, these torsor laws extend
to semitorsor laws onto the whole projective geometry, in the same way as the group law
of $\Bij(X)$ for a set $X$ extends to a semigroup structure on $\Map(X,X)$:

\begin{theorem}\label{TheoremB}
Assume $(a,b)$ is a pair of  subgroups of $\Omega$. 
Then the following holds:
\begin{enumerate}
\item
The map $(x,y,z) \mapsto (xyz)_{b}$ defines a semitorsor structure on $\cP$.
\item
The map $(x,y,z) \mapsto (xyz)_{ab}$ defines a semitorsor structure on $\cP$.
\end{enumerate}
\end{theorem}

\nin
We call the torsors $U_{ab}$ {\em balanced}, and the torsors $U_b$ {\em unbalanced}.
If $\Omega$ is non-abelian, the torsor $U_{bb}$ is different from $U_b$ -- the latter
are in general not members of a two-parameter family. This is due to the fact that
  the system of three equations defining $\Gamma$, called the {\em structure equations},
\begin{equation} \label{EquationS0}
\left\{
\begin{matrix}
\zeta &=& \alpha + \omega \cr
\eta & =& \alpha + \omega + \beta \cr
\xi &=& \omega + \beta 
\end{matrix} \right\}
\end{equation}
is of a more symmetric nature than the one defining $\Sigma$. 
We come back to this item below (Subsection \ref{Subsec:Symmetry}).

\subsection{Affine picture}
As usual in projective geometry, a ``projective statement'' may be translated into
an ``affine statement'' by choosing some ``affinization'' of $\cP$.
Thus one can rewrite the torsor law of $U_{ab}$ by an ``affine formula''
(Theorem \ref{PictureTheorem}). Here is a quite instructive special case:
consider two arbitrary groups, $(V,+)$ and $(W,+)$,  and fix a group homomorphism 
$A:W \to V$.
Let $G:=\Map(V,W)$ be the set of all maps from $V$ to $W$. Then it is an easy 
exercise to show that
\begin{equation}\label{AffEqn1}
X  \cdot_A   Y := Y + X\circ (\id_V  +  A  \circ Y )
\end{equation}
defines an associative product on $G$ (where $+$ is the pointwise ``sum''
of maps), with neutral element the ``zero map'' $0$, and
which gives a group law on the set
\begin{equation}\label{AffEqn2}
G_A:=
\{X \in G \mid \,  \id_V + A \circ X  \mbox{ is bijective } \}  .
\end{equation}
The parameter $A$ is a sort of ``deformation parameter'': if $A=0$, we get
pointwise addition; if $A$ is an isomorphism, then $G_A$ is in fact isomorphic
to the usual group $\Bij(V)$. 
If $V=W=\bR^n$, then one may do the same construction using continuous, or smooth,
maps, and thus gets a deformation of the abelian (additive) group $G$ of vector fields to the
highly non-commutative group $G_A$ of diffeomorphisms of $\bR^n$.
If $V,W$ are linear spaces and $X,Z$ linear maps, then (\ref{AffEqn1}) gives us back the 
law $X+ XAY + Y$ considered in \cite{BeKi10}.

\subsection{Distributive laws and near-rings.}
The reason why we are also interested in the unbalanced torsors is that
they are natural (being spaces of sections of a principal bundle over a homogeneous
space) and interact nicely with the balanced structures: there is a 
base-point free version of a {\em right
distributive law} which makes the whole object a torsor-analog of 
 a {\em near-ring} or a ``generalized ring'' (cf.\  \cite{Pi77}):

\begin{definition}\label{NearringDefinition}
A {\em (right) near-ring} is a set $N$ together with two binary operations, denoted by $+$ and
$\cdot$, such that:
\begin{enumerate}
\item $(N,+)$ is a group (not necessarily abelian),
\item $(N,\cdot)$ is a semigroup,
\item we have the {\em right distributive law} $(x+y)\cdot z = x\cdot z + y \cdot z$.
\end{enumerate}
\end{definition}

\nin A typical example is the set $N$ of self-maps of a group $(G,+)$, 
where $\cdot$ is composition and $+$  pointwise ``addition''. 
In our context, $\Gamma$ takes the role of the product $\cdot$, and $\Sigma$ takes
the one of the ``addition'' $+$ (cf.\ Theorem \ref{DistributiveTheorem}):

\begin{theorem}
Let $(a,b)$ be a pair of subgroups of $\Omega$. Then we have the following
{\em left distributive  law} relating the unbalanced and the
balanced torsor structures: for all $x,y \in U_{ab}$ and $u,v,w \in U_b$,
$$
(xy (uvw)_b )_{ab} = \bigl( (xyu)_{ab} (xyv)_{ab} (xyw)_{ab} )_{b} .
$$
\end{theorem}

\nin Essentially, this means that $ { }^\top b$ looks like a ternary version of a near-ring,
whose ``multiplicative'' structure now depends on an additional parameter $y$.
As usual for near-rings, there is just one distributive law: the other distributive law
does not hold!
Compare with (\ref{AffEqn1}) which is ``affine'' in $X$, but not in $Y$.

\subsection{Symmetry}\label{Subsec:Symmetry}
The preceding theorem makes it obvious that the definition of $\Gamma$ involves
some arbitrary choices: there is no reason why left distributivity should be 
preferred to right distributivity!
Indeed, if instead of $\Gamma$, we looked at the map $\check \Gamma$ obtained
by using everywhere the opposite group law of $(\Omega,+)$, then we would get
``right'' instead of ``left distributivity''.
Thus $\Gamma$ and $\check \Gamma$ are, in a certain sense, ``equivalent''.
In the same way, there is no reason to prefer the groups $a$ or $b$ to their
opposite groups: in the structure equations we might replace $\alpha$ or $\beta$
by their negatives, without changing the whole theory.
Thus we are led to consider several versions of the fundamental equations as
``essentially equivalent''. 
We investigate this item in Section \ref{Sec:Symmetry}: there are in fact 24 
{\em signed}  (i.e., essentially equivalent) versions of the structure equations on
which a certain subgroup $\bV$ of the permutation group $S_6$, permuting  the six
variables of (\ref{EquationS0}),  acts simply transitively (Theorem \ref{BigKleinTheorem}).
We call $\bV$ the {\em Big Klein Group}\footnote{translated from the
German {\em Grosse Klein Gruppe}} since it plays exactly the same r\^ole for the
structure equations as the usual Klein Group $V$ does for a single torsor
structure (cf.\ Lemma   \ref{TorsorLemma}). 
The group $\bV$ is isomorphic to $S_4$, sitting inside $S_6$ as the subgroup 
preserving  the partition of six letters in three subsets
$\{ \xi, \zeta \}$, $\{\alpha ,\beta \}$ and $\{ \eta,\omega \}$ (Lemma \ref{BigKleinLemma}).
Permutations from $\bV$ leave invariant the general shape of (\ref{EquationS0})
and introduce just certain sign changes for some of the variables.
If we are willing to neglect such sign changes -- like, for instance, in the ``projective'' framework of
\cite{BeKi10}, where one can rescale by any invertible scalar -- then the whole
theory becomes invariant under these permutations.\footnote{A side remark: 
the author cannot help feeling being reminded by this situation to
CPT-invariance in physics, where a very similar phenomenon occurs.
} 
 This explains partially why
the associative geometries from \cite{BeKi10} (and their Jordan theoretic analogs)
have such a high degree of symmetry (cf.\ the ``symmetry'' and ``duality principles''
for Jordan theory, \cite{Lo75}).
If we agree to neglect sign changes only with respect to $\alpha$ and $\beta$
(which is reasonable since in Theorem \ref{TheoremA} we assume that $a$ and $b$
are {\em subgroups}, hence $\alpha \in a$ iff $-\alpha \in a$, and
same for $b$), then we obtain as invariance group again a usual
Klein Group $V$, and the orbit under $\bV$ has $24/4 = 6$ elements.
This, in turn, is completely analogous to the behavior of the classical
{\em cross-ratio} under $S_4$, which is invariant under $V$ and takes generically
$6$ different values under permutations.

\subsection{Further topics}
Because of its generality, the approach presented in this work is likely to interact
with many other mathematical theories. In the last section we mention some
questions arising naturally in this context, and we refer to Section 4 of \cite{BeKi10}
for some more remarks of a similar kind.


\begin{notation}
Throughout this paper, $\Omega$ is a (possibly non-commutative) group,
called the {\em background},  whose
group law will be written additively. 
Its neutral element will be denoted by $o$.
We denote by $\cP := \cP(\Omega)$ its power set, by
$\cP^o:=(\Omega)$ the set of subsets of $\Omega$ containing the neutral element $o$,
and by $\Gras(\Omega)$ the Grassmannian of $\Omega$ (the set of all subgroups of $\Omega$).
{\em Transversality}, as defined in Definition \ref{TransversalDefinition} above, is denoted 
by $x \top y$. 
\end{notation}

%
%
%

\section{Structure maps and structure space}

\begin{definition}\label{StructureDefinition} The {\em structure maps} of a group $(\Omega,+)$ are
the maps $\Gamma:\cP^5 \to \cP$, $\check \Gamma:\cP^5 \to \cP$,
$\Sigma:\cP^4 \to \cP$ and $\check\Sigma:\cP^4 \to \cP$ defined, for
 $x,a,y,b,z \in \cP$, by
\begin{align}
\Gamma(x,a,y,b,z) &:=
\Bigsetof{\omega \in \Omega}
{\begin{array}{c}
\exists \xi \in x,
\exists \alpha \in a,
\exists \eta \in y,
\exists \beta \in b,
\exists \zeta \in z : \\
\eta =  \alpha +  \omega + \beta,
\zeta = \alpha + \omega,
\xi = \omega  + \beta
\end{array}} ,
\end{align}
\begin{align}
\check \Gamma(x,a,y,b,z) &:=
\Bigsetof{\omega \in \Omega}
{\begin{array}{c}
\exists \xi \in x,
\exists \alpha \in a,
\exists \eta \in y,
\exists \beta \in b,
\exists \zeta \in z : \\
\eta =  \beta + \omega + \alpha,
\zeta = \omega + \alpha,
\xi = \beta + \omega  
\end{array}} \, ,
\end{align}
\begin{align}
\Sigma(a,x,y,z) &:=
\Bigsetof{\omega \in \Omega}
{\begin{array}{c}
 \exists \xi \in x,\exists \eta \in y,\exists \zeta \in z,\exists \beta,\beta' \in b: \\
\xi = \omega + \beta, \,
\eta = \omega + \beta' + \beta, \, 
\zeta = \omega +\beta' 
\end{array}} \, ,
\end{align}
\begin{align}
\check \Sigma(a,x,y,z) &:=
\Bigsetof{\omega \in \Omega}
{\begin{array}{c}
 \exists \xi \in x,\exists \eta \in y,\exists \zeta \in z,\exists \beta,\beta' \in b: \\
\xi = \beta + \omega , \,
\eta = \beta + \beta' + \omega , \, 
\zeta = \beta' + \omega  
\end{array}} \, .
\end{align}
It is obvious that the set $\cP^o(\Omega)$ of subsets containing $o$ is stable under each of these maps,
and the corresponding restrictions of the four maps will also be called {\em structure maps}.
\end{definition}

\nin Note that $\check \Gamma$, resp.\ $\check \Sigma$,
 is obtained from $\Gamma$, resp.\ $\Sigma$ simply by replacing the group
law in $\Omega$ by the opposite group law.
Hence, if $\Omega$ is abelian, we have $\Gamma = \check \Gamma$ and
$\Sigma = \check \Sigma$.
Moreover, if $\Omega$ is abelian,  we obviously have
\begin{equation}
\Gamma(x,a,y,a,z) = \Sigma(a,x,y,z)=\Sigma(a,z,y,x) .
\end{equation}
For general $\Omega$,
the defining equations immediately imply the {\em symmetry relation}
\begin{equation} \label{Sym1}
\check \Gamma(z,b,y,a,x) = \Gamma(x,a,y,b,z) .
\end{equation}

\begin{definition}\label{StructurespaceDefinition}
The system (\ref{EquationS0}) of three equations for six variables in $\Omega$
is called  the {\em structure equations}. 
We say that another system of equations is {\em equivalent} to the structure equations if
it has the same set of solutions,  called
the {\em structure space} of the group $(\Omega,+)$:  
$$
\Gamm:=
\Bigsetof{ (\xi,\zeta; \alpha,\beta; \eta,\omega) \in \Omega^6}
{\eta =  \alpha +  \omega + \beta, \quad
\zeta = \alpha + \omega, \quad
\xi = \omega  + \beta }  \, .
$$
By definition, the {\em opposite structure space} is the structure space of $\Omega^{opp}$:
$$
\check \Gamm:=
\Bigsetof{ (\xi,\zeta; \alpha,\beta; \eta,\omega) \in \Omega^6}
{\eta =  \beta +  \omega + \alpha, \quad
\zeta = \omega + \alpha,\quad
\xi = \beta  + \omega } .
$$
The sets ${\mathbf \Sigma} \subset \Omega^5$ and $\check{\mathbf \Sigma} \subset \Omega^5$ 
can be defined similarly.
\end{definition}

\begin{lemma}\label{SymmetryLemma}
The following systems are all equivalent to the structure equations:
\begin{equation}\label{2.3}
\left\{
\begin{matrix}
\alpha &=& \eta - \xi \cr
\omega &=& \xi - \eta + \zeta \cr
\beta &=& - \zeta + \eta 
\end{matrix} \right\}
\end{equation}

\begin{equation}\label{EquationS2}
\left\{
\begin{matrix}
\eta &=& \alpha + \omega + \beta \cr
\eta &=& \alpha +  \xi \cr
\eta &=& \zeta + \beta
\end{matrix}
\right\} 
\qquad \qquad
\left\{
\begin{matrix}
\eta &=& \zeta - \omega + \xi \cr
\eta &=& \alpha +  \xi \cr
\eta &=& \zeta + \beta
\end{matrix}
\right\}
\end{equation}

\begin{equation}\label{EquationS3}
\left\{
\begin{matrix}
\omega & = &  \xi - \eta + \zeta  \cr
\omega & = & \xi - \beta  \cr
\omega  & = & - \alpha + \zeta
\end{matrix} 
\right\}
\qquad \qquad
\left\{
\begin{matrix}
\omega & = &  -\alpha  + \eta  - \beta  \cr
\omega & = & \xi - \beta  \cr
\omega  & = & - \alpha + \zeta
\end{matrix} 
\right\}
\end{equation}

\begin{equation}\label{EquationS4}
\left\{
\begin{matrix}
\alpha &=& \zeta + \beta - \xi \cr
\alpha &=& \eta - \xi \cr
\alpha &=& \zeta - \omega 
\end{matrix}
\right\}
\qquad \qquad
\left\{
\begin{matrix}
\alpha &=& \eta - \beta - \omega \cr
\alpha &=& \eta - \xi \cr
\alpha &=& \zeta - \omega 
\end{matrix}
\right\}
\end{equation}

\begin{equation}\label{EquationS5}
\left\{
\begin{matrix}
\beta & = &  - \zeta + \alpha  + \xi  \cr
\beta & = &  - \omega + \xi  \cr
\beta  & = & - \zeta + \eta 
\end{matrix} 
\right\}
\qquad \qquad
\left\{
\begin{matrix}
\beta & = &  - \omega  -  \alpha  + \eta  \cr
\beta & = &  - \omega + \xi  \cr
\beta  & = & - \zeta + \eta 
\end{matrix} 
\right\}
\end{equation}

\begin{equation}\label{EquationS6}
\left\{
\begin{matrix}
\xi & = &  - \alpha + \zeta + \beta  \cr
\xi & = &  - \alpha + \eta   \cr
\xi  & = &  \omega + \beta
\end{matrix} 
\right\}
\qquad \qquad
\left\{
\begin{matrix}
\xi & = &  \omega - \zeta + \eta  \cr
\xi & = &  - \alpha + \eta   \cr
\xi  & = & \omega + \beta
\end{matrix} 
\right\}
\end{equation}

\begin{equation}\label{EquationS7}
\left\{
\begin{matrix}
\zeta & = &  \eta -\xi  + \omega  \cr
\zeta & = &  \eta - \beta  \cr
\zeta  & = &  \alpha + \omega 
\end{matrix} 
\right\}
\qquad \qquad
\left\{
\begin{matrix}
\zeta & = &  \alpha + \xi  - \beta  \cr
\zeta & = &  \eta - \beta  \cr
\zeta  & = &  \alpha + \omega 
\end{matrix} 
\right\}
\end{equation}

\begin{equation}\label{EquationS8}
\left\{
\begin{matrix}
\eta  &=& \alpha + \xi \cr
\beta  &=& - \omega + \xi \cr
\zeta &=& \alpha + \omega
\end{matrix}
\right\}
\qquad
\left\{
\begin{matrix}
\alpha &=& \eta - \xi \cr
\zeta  &=& \eta - \beta \cr
\omega  &=& \xi - \beta
\end{matrix}
\right\}
 \qquad
\left\{
\begin{matrix}
\alpha &=& \zeta - \omega \cr
\xi   &=& \omega +  \beta \cr
\eta  &=& \zeta + \beta
\end{matrix}
\right\}
\end{equation}
\end{lemma}

\nin The proof is by completely elementary computations. 
Obviously, the structure space has certain 
symmetry properties with respect to permutations. This will be investigated
in more detail in Section \ref{Sec:Symmetry} .
Note also that,
if $\Omega$ is abelian, the structure equations are {\em $\Z$-linear}, and hence can
be written in matrix form
$$
\begin{pmatrix} 1 & 1 & 0 \cr 1 & 1 & 1 \cr 0 & 1 & 1 \end{pmatrix}
\begin{pmatrix} \alpha \cr \omega \cr \beta \end{pmatrix}=
\begin{pmatrix} \zeta \cr \eta \cr \xi \end{pmatrix} \, .
$$
Equations (\ref{2.3}) then correspond to the inverse of this matrix. 

\section{The semitorsor laws}\label{Sec:Semitorsor}

\begin{theorem}\label{SemitorsorTheorem}
Assume that $a$ and $b$ are two subgroups of a  group $(\Omega,+)$.
Then the power set $\cP$ and its subset $\cP^o$ become semitorsors under the ternary
compositions
\begin{align*}
\cP^3 \to \cP, & \quad (x,y,z) \mapsto (xyz)_{ab}:=\Gamma(x,a,y,b,z), \cr
\cP^3 \to \cP, & \quad (x,y,z) \mapsto (xyz)_{ab}\check{ }:=\check \Gamma(x,a,y,b,z) ,\cr
\cP^3 \to \cP, & \quad (x,y,z) \mapsto (xyz)_b:=\Sigma(b,x,y,z),\cr
\cP^3 \to \cP, & \quad (x,y,z) \mapsto (xyz)_b\check{ }:=\check \Sigma(b,x,y,z) .
\end{align*}
We denote these semitorsors by $\cP_{ab}$, $\check \cP_{ab}$, $\cP_b$, $\check \cP_b$, 
respectively.
\end{theorem}

\begin{proof}  We prove,
for $x,y,z \in \cP(\Omega)$, the identity 
\[
\Gamma\bigl(x,a,u,b,\Gamma(y,a,v,b,z)\bigr) =
\Gamma\bigl(x,a,\Gamma(v,a,y,b,u),b,z\bigr) =
\Gamma\bigl(\Gamma(x,a,u,b,y),a,v,b,z\bigr)\, ,
\]
i.e., the semitorsor law for $(xyz)_{ab}$. For the proof, note that
the definition of $\Gamma(x,a,y,b,z)$ can be written somewhat shorter, as follows:
\begin{align}
\Gamma(x,a,y,b,z) &=
\Bigsetof{\omega \in \Omega}
{\begin{array}{c}
\exists \alpha \in a,
\exists \beta \in b: \\
 \alpha +  \omega + \beta \in y, \quad
\alpha + \omega \in z, \quad
\omega  + \beta \in x
\end{array}} ,
\end{align}
and similarly for $\check \Gamma$. 
We refer to this description as {\em $(a,b)$-description}.
Using this, we have,
on the one hand, 

\smallskip
\noindent$\Gamma\big(x,a,u,b,\Gamma(y,a,v,b,z)\big) =$
\begin{align*}
&= \Bigsetof{ \omega \in \Omega}{
\begin{array}{c}
\exists \alpha \in a, \exists \beta \in b : \\
\alpha + \omega  \in \Gamma(y,a,v,b,z), \alpha + \omega + \beta \in u,
\omega  + \beta \in x
\end{array} } \\
&= \Bigsetof{ \omega \in \Omega}{
\begin{array}{c}
\exists \alpha \in a, \exists \beta \in b,
\exists \alpha' \in a, \exists \beta' \in b : \\
\alpha + \omega + \beta \in u,
\omega + \beta \in x,
\alpha' + \alpha + \omega  \in z, \\
\alpha' + \alpha + \omega + \beta' \in v,
\alpha + \omega + \beta' \in y
\end{array} } \ .
\end{align*}
On the other hand,

\smallskip
\noindent$\Gamma\big(x,a,\Gamma(v,a,y,b,u),b,z\big) =$
\begin{align*}
&= \Bigsetof{ \omega \in \Omega}{
\begin{array}{c}
\exists \alpha'' \in a, \exists \beta'' \in b : \\
\alpha'' + \omega \in z,  \alpha'' + \omega + \beta'' \in \Gamma(v,a,y,b,u),
\omega  + \beta'' \in x
\end{array} } \\
&= \Bigsetof{ \omega \in \Omega}{
\begin{array}{c}
\exists \alpha'' \in a, \exists \beta'' \in b,
\exists \alpha''' \in a, \exists \beta''' \in b : \\
\alpha'' + \omega \in z,\omega + \beta'' \in x,
\alpha''' + \alpha'' + \omega + \beta'' \in u, \\
\alpha'' + \omega + \beta'' + \beta'''  \in v,
\alpha''' + \alpha'' + \omega + \beta'' + \beta'''   \in y
\end{array} } \ .
\end{align*}
Via the change of variables
$\alpha''=\alpha' + \alpha$,
$\alpha''' = \alpha'$,
$\beta''=\beta$,
$\beta'''=-\beta + \beta'$, we see that these two subsets of $\Omega$ are the same.
(Here we use that $a$ and $b$ are groups!)
This proves the first defining equality of a semitorsor for $\Gamma$.
Since $\Omega^{opp}$ is again a group, it holds also for $\check \Gamma$.
The second equality now follows from the first one, using the symmetry relation (\ref{Sym1}).

\ssk
Now consider the product $(xyz)_b$. 
Similarly as above, we have
\begin{align}\label{SigmaEqn1}
\Sigma(b,x,y,z) &=\Bigsetof{\omega \in \Omega}
{\begin{array}{c}
 \exists \beta,\beta' \in b: \\
 \omega + \beta \in x , \, \omega +\beta' + \beta \in y , \, \omega + \beta' \in z
\end{array}} 
\end{align}
Using (\ref{SigmaEqn1}), we have on the one hand, 

\smallskip
\noindent$(x,u,(y,v,z)_b)_b =$
\begin{align*}
&= \Bigsetof{ \omega \in \Omega}{
\begin{array}{c}
\exists \alpha \in b, \exists \beta \in b : \\
\omega + \alpha  \in (y,v,z)_b, \omega + \alpha + \beta \in u,
\omega  + \beta \in x
\end{array} } \\
&= \Bigsetof{ \omega \in \Omega}{
\begin{array}{c}
\exists \alpha \in b, \exists \beta \in b,
\exists \alpha' \in b, \exists \beta' \in b : \\
\omega + \alpha + \beta \in u,
\omega + \beta \in x,
 \omega + \alpha + \alpha'  \in z, \\
\omega + \alpha + \alpha' + \beta' \in v,
 \omega + \alpha +  \beta' \in y
\end{array} } \ .
\end{align*}
On the other hand,

\smallskip
\noindent$\big(x,(v,y,u)_b z\big)_b =$
\begin{align*}
&= \Bigsetof{ \omega \in \Omega}{
\begin{array}{c}
\exists \alpha'' \in b, \exists \beta'' \in b : \\
\omega + \alpha''  \in z,  \omega + \alpha'' + \beta'' \in (v,y,u)_b,
\omega  + \beta'' \in x
\end{array} } \\
&= \Bigsetof{ \omega \in \Omega}{
\begin{array}{c}
\exists \alpha'' \in b, \exists \beta'' \in b,
\exists \alpha''' \in b, \exists \beta''' \in b : \\
\omega + \alpha''  \in z,\omega + \beta'' \in x,
\omega + \alpha''  + \beta'' + \alpha'''  \in u, \\
\omega + \alpha''  + \beta'' + \beta'''  \in v,
\omega  + \alpha'' + \beta'' + \alpha'''+ \beta'''   \in y
\end{array} } \ . 
\end{align*}
Via the change of variables
$\beta = \beta''$,
$\beta' = \beta'' + \beta'''$,
$\alpha = \alpha'' + \beta'' + \alpha''' - \beta'''$,
$\alpha'=\beta'' - \alpha''' - \beta''$, we see that these two subsets of $\Omega$ are the same.
This proves the first defining equality of a semitorsor for $(xyz)_b$.
The proof of the other defining equality, as well as for the semitorsor structure
$(xyz)_b\check{ }$, are similar.
\end{proof}

\begin{definition}
We call the semitorsors $\cP_{ab}, \check \cP_{ab}$  {\em balanced} and 
 $\cP_b, \check \cP_b$ {\em unbalanced}.
\end{definition}

\nin By the symmetry relation (\ref{Sym1}), $\check \cP_{ba}$ is the opposite semitorsor 
of $\cP_{ab}$ (where, for any semitorsor $(xyz)$, the {\em opposite} law is just
$(zyx)$), whereas $\check \cP_b$ is {\em not} the opposite semitorsor of $\cP_b$. 
Thus, given a subgroup $b\subset \Omega$,
we have in general six different semitorsor laws on $\cP$:
$\cP_{bb}$, $\cP_b$, $\check \cP_b$, along with their opposite laws. 
If $\Omega$ is commutative, then of course these six semitorsor  laws coincide.
More generally:

\begin{theorem}\label{SubgroupTheorem}
Assume that $a$ and $b$ are {\em central} subgroups of $\Omega$.
Then:
\begin{enumerate}
\item
$\cP_{ab} = \check \cP_{ab}=\cP_{ba}^{opp}$ and
$\cP_b = \check \cP_b = \check \cP^{opp} = \cP_b^{opp}$,
\item
$\Gras(\Omega)$ is stable under all ternary laws from
Theorem \ref{SemitorsorTheorem}.
\end{enumerate}
\end{theorem}

\begin{proof}
The first statement follows immediately from the definitions, and the second by writing the
structure equations for $\omega + \omega'$, resp.\ for $-\omega$,  with
$\omega,\omega' \in \Gamma(x,a,y,b,z)$, and using that variables from $a$ and $b$
commute with the others.
\end{proof}

Note that our condition  is sufficient, but not necessary with respect to item (2): 
for instance, if $\Omega$ is a direct product of $a$ and $b$ (as a group), 
then the result of the next section implies that
$\Gras(\Omega)$ is a subsemitorsor of $\cP_{ab}$. 
For general subgroups $a,b$, this is no longer true:
the subsemitorsor generated by $\Gras(\Omega)$ will be strictly bigger.

\section{The transversal case:  composition of relations in groups}\label{sec:Composition}

Recall the definition of {\em (left) transversality} (Definition \ref{TransversalDefinition}), denoted
by $a \top b$. For a fixed transversal pair, we may identify $\Omega$ as a set with $a \times b$ 
via $(\alpha,\beta) \mapsto \alpha + \beta$. 
Then (by definition)  the power set
$\cP(\Omega)$ is identified with the set $\Rel(a,b)$ of {\em relations between $a$ and $b$}.

\begin{theorem}\label{CompositionTheorem}
Let $(a,b)$ be a pair of left transversal subgroups of  $\Omega$.
Then the ternary {\em composition
$z \circ y\inv \circ x$ of  relations} $x,y,z \in \cP = {\Rel}(a,b)$ is given by
$$
z\circ y^{-1} \circ x = \Gamma(x,a,y,b,z) \, .
$$
If $a$ and $b$ commute, then $ \Gras(\Omega)$  is stable
under this ternary law.
\end{theorem}


\begin{proof}
The computation is the same as in \cite{BeKi10}, Lemma 2.1, by respecting the possible 
non-commutativity of $\Omega$.
Recall first  that, if $A,B,C,\ldots$ are any sets,
we can \emph{compose
relations}: for
$x \in \Rel(A, B)$, $y \in \Rel( B ,C)$,
\[
y \circ x := yx :=  \setof{(u,w) \in A \times C}
{\exists v \in B: \, (u,v) \in x, (v,w) \in y}\,.
\]
Composition is associative: both $(z \circ y) \circ x$ and
$z \circ (y \circ x)$ are equal to
\begin{equation}
z \circ y \circ x = \setof{(u,w) \in A \times D}
{\exists (v_1,v_2) \in y: \,
(u,v_1) \in x, (v_2,w) \in z} \,.
\label{assoc}
\end{equation}
The \emph{reverse relation} of $x$ is
$$
x\inv := \setof{(w,v) \in B \times A}{(v,w) \in x}.
$$
For $x,y,z \in \Rel(A, B)$, we get another relation between $A$ and $B$ by
$zy^{-1}x $. 
(This ternary composition satisfies the para-associative law, and hence
relations between sets $A$ and $B$ form
a semitorsor; no structure on the sets $A$ or $B$ is needed here.) 
Coming back to $\Omega=a \times b$, and switching to an additive notation,
we get
\begin{align*}
z\circ y^{-1} \circ x &=
\Bigsetof{\omega = (\alpha',\beta') \in \Omega}
{\begin{array}{c}
\exists \eta=(\alpha'',\beta'') \in y:
 \\
(\alpha',\beta'') \in x, (\alpha'',\beta') \in z
\end{array}}\,
\\
& =
\Bigsetof{\omega  \in \Omega}
{\begin{array}{c}
\exists \alpha',\alpha'' \in a, \exists \beta',\beta'' \in b, \,  \exists
\eta \in y, \exists
\xi \in x, \exists \zeta \in z:
 \\
\omega = (\alpha',\beta'),
\eta = (\alpha'',\beta''),
\xi=(\alpha',\beta''),
\zeta = (\alpha'',\beta')
\end{array}}\,
\\
& =
\Bigsetof{\omega  \in \Omega}
{\begin{array}{c}
\exists \alpha',\alpha'' \in a, \exists \beta',\beta'' \in b, \exists
\eta \in y,
\exists \xi \in x, \exists \zeta \in z:
 \\
\omega = \alpha'  + \beta',
\eta = \alpha''  + \beta'',
\xi=\alpha'   + \beta'',
\zeta = \alpha'' +  \beta'
\end{array}}\, .
\end{align*}
Now use that $a$ and $b$ are transversal {\em subgroups} of $\Omega$:
then the description of $zy^{-1}x$ can
be rewritten, by introducing the new variables
$\alpha:=\alpha' - \alpha''$,
$\beta:=- \beta'' + \beta'$ (which belong again to $a$, resp.\ to $b$, since these are
subgroups)
\begin{align*}
zy^{-1}x & =
\Bigsetof{\omega  \in \Omega}
{\begin{array}{c}
\exists \alpha',\alpha \in a, \exists \beta',\beta \in b, \exists \eta \in y,
\exists \xi \in x, \exists \zeta \in z:
 \\
\omega = \alpha' +  \beta',
\eta = - \alpha + \omega - \beta,
\xi=\omega - \beta,
\zeta =  - \alpha +  \omega
\end{array}}\, .
\end{align*}
Since $a\top b$, the first condition ($\exists
\alpha' \in a, \beta' \in b$: $\omega = \alpha'  + \beta'$)
in the preceding description is always satisfied and can hence
be omitted in the description of $zy\inv x$.
Thus 
$$
zy\inv x =
\Bigsetof{\omega  \in \Omega}
{\begin{array}{c}
\exists \alpha',\alpha \in a, \exists \beta',\beta \in b, \exists \eta \in y,
\exists \xi \in x, \exists \zeta \in z:
 \\
\eta = - \alpha + \omega - \beta,
\xi=\omega - \beta,
\zeta =  - \alpha +  \omega
\end{array}} = \Gamma(x,a,y,b,z) \, .
$$
Finally, if $a$ and $b$ commute, then the bijection $\Omega \cong a \times b$ is also
a group homomorphism. Since subgroups in a direct product of groups 
form a monoid under composition of relations, it follows that $\Gras(\Omega)$ is stable under the ternary composition map.
\end{proof}

Recall that maps give rise to relations via their  {\em graphs}.  In our setting:

\begin{definition}
Assume $(x,y)$  is a left-transversal pair of subsets of $\Omega$: $x \top y$, and let
 $F:x \to y$ be a map. The {\em (left) graph of $F$} is the subset
$$
G_F := \{ \xi + F(\xi) \vert \, \xi \in x \} \subset \Omega,
$$
and if $y \top x$, we define the {\em right graph} of $F:x \to y$ to be
$$
\check G_F:= \{ F(\xi) + \xi \vert \, \xi \in x \} \subset \Omega.
$$
\end{definition}

\begin{lemma}\label{GraphtransversalLemma}
Let $b$ be a subgroup of $\Omega$, and $y$ a subset such that $y \top b$.
Then there are natural bijections between the following sets:
\begin{enumerate}
\item
the set
${ }^\top b$,
\item
the set of sections $\sigma:\Omega/b \to \Omega$ of the canonical projection 
$\pi:\Omega \to \Omega/b$, 
\item
the set $\Map (y,b)$ of maps from $y$ to $b$.
\end{enumerate}
More precisely, the bijection between (1) and (2) is given by
the correspondence between $\sigma$ and the image of $\sigma$, and the one between
(1) and (3) by
$\Map (y,b) \to { }^\top b$, $F \mapsto G_F$.
If, moreover, $y$ is a subgroup, then we have:

\ssk
${ }\, $ (B)${ }$
the map $F$ is bijective iff $y \top G_F$. 
\end{lemma}

\begin{proof} 
Consider the equivalence relation given on $\Omega$ by 
$\omega \sim \omega'$ iff $-\omega + \omega' \in b$. 
Then $x \top b$ if, and only if, $x$ is a set of representatives for this equivalence
relation. As for any equivalence relation, it follows therefore that
${ }^\top b$ is in bijection with the set of sections of the canonical projection
$\Omega \to \Omega/\sim$. 
Now let $y = \sigma(\Omega/b)$ and
$x = \sigma'(\Omega/b)$ for two sections $\sigma,\sigma'$.
Then $f:=-\sigma + \sigma'$ is a map $\Omega/b \to b$.
Conversely, given $f:\Omega/b \to b$, $\sigma':=\sigma + f$ is another section,
whose image is precisely the left graph of the map
$F:=f \circ \pi\vert_y:y \to b$. 

For the last statement, assume that $y$ is a subgroup and that
 $F$ is bijective. If $\omega = \eta + \beta$ with $\beta = F(\eta')$, we have the decomposition
$\omega = \eta - \eta' + \eta' + F(\eta')$ with $\eta - \eta' \in y$ and
$\eta' + F(\eta') \in G_F$ which is unique since $F$ is injective. Hence $y \top G_F$.
The converse is proved similarly.
\end{proof}

In a similar way, if $b  \top y$, every element of $b^\top$ is of the form $\check G_F$ with a unique
map $F:y \to b$.

\begin{theorem}\label{BijectionTorsorTheorem}
Let  $(a,b)$ be a pair of left-transversal subgroups of a group $\Omega$.
Then $a^\top \cap { }^\top b$ is a subsemitorsor of $\cP_{ab}$, and it is actually a torsor,
denoted by $U_{ab}$ and
naturally isomorphic to the torsor of bijections $F: a \to b$ with 
its usual torsor law
$$
(XYZ) = Z \circ Y\inv \circ X .
$$
In other words, if one fixes a bijection $Y:a\to b$ in order to identify $a$ and $b$, then
$U_{ab}$ is the torsor corresponding to the group of all bijections of $a$. \footnote{To be
precise, we get the opposite of the ``usual'' composition. This convention, used already
in \cite{BeKi10}, is in keeping with certain formulas from Jordan theory. 
}
\end{theorem}

\begin{proof}
By the lemma, 
a relation $r \in \cP$ belongs to $a^\top \cap { }^\top b$ if, and only if, it is the left graph
of a bijection $F:a \to b$. Since composition of maps corresponds precisely to the
composition of their graphs, the claim now follows from Theorem 
\ref{CompositionTheorem}.
\end{proof} 

\begin{definition}
A {\em transversal triple of subgroups} is a triple of subgroups $(a,b,c)$ of $\Omega$
such that $a$ and $b$ commute, and $a \top b$, $b \top c$ and $c \top a$.
\end{definition}

\begin{theorem}\label{TripleTheorem}
Let $(a,b,c)$ be a triple of subgroups such that $a \top b$ and $a,b$ commute. 
Then $(a,b,c)$ is a  transversal triple if and only if
$\Omega \cong a \times a$, with $a$ the first, $b$ the second factor and $c$ the
diagonal. The subset
$U_{ab}':=U_{ab} \cap \Gras(\Omega)$ of $U_{ab}$ is a torsor with base point $c$, 
isomorphic to the  group $\Aut(a)$ of group automorphisms of $a$.
\end{theorem}

\begin{proof}
The assumption implies that $\Omega \cong a \times b$ as a group. Let $c \in U_{ab}$ be the
graph of the bijective map $F:b \to a$.
Then $c\subset \Omega$
 is a subgroup if and only if $F$ is a group morphism, and the claim follows
from the preceding results.
\end{proof}

\section{The singular case : pointwise torsors}


Now we turn to the ``singular'' case $a=b$. In this case, $\Gamma$ has to be replaced by
$\Sigma$, and torsors of bijections by ``pointwise torsors'':

\begin{theorem}\label{PointwisetorsorTheorem}
Assume $b$ is a subgroup of $\Omega$, and let $y \in \cP$ such that $y \top b$. 
Then there are natural isomorphisms between the following torsors:
\begin{enumerate}
\item
the set
${ }^\top b$,
which is a subsemitorsor of $\cP_b$,  and which becomes a torsor, denoted by $U_b$,
 with the induced law, 
\item
the torsor
 of all (images) of sections $\sigma:\Omega/b \to \Omega$ of the canonical projection 
$\Omega \to \Omega/b$ with  pointwise torsor structure
$(\sigma \sigma' \sigma'') (u) = \sigma(u) - \sigma'(u)+\sigma''(u)$,
\item
the torsor $\Map (y,b)$ of maps from $y$ to $b$, with its pointwise torsor structure.
\end{enumerate}
Similar statements hold for $\check U_b=b^\top$, which can be identified with sections of the projection
$\Omega \to b \backslash \Omega$, together with pointwise torsor structure.
\end{theorem}

\begin{proof}
On the level of sets, these bijections  have been established in Lemma 
\ref{GraphtransversalLemma}.
It is  immediately checked that 
the set of sections of the projection is stable under the pointwise torsor structure
$(\sigma \sigma' \sigma'') (u) = \sigma(u) - \sigma'(u)+\sigma''(u)$ (as well as
under its opposite torsor structure), and
it is clear, then, that the bijection between (2) and (3) becomes an
isomorphism of torsors with pointwise torsor structures. 

\ssk
In order to show that these torsor structures agree with the law described in Theorem
 \ref{SemitorsorTheorem}, note that,
by a change of variables, the unbalanced semitorsor law can also be written
\begin{align}\label{NonbalancedEquation}
(xyz)_b &=
\Bigsetof{\omega \in \Omega}
{\begin{array}{c}
 \exists \xi \in x,\exists \eta \in y,\exists \zeta \in z,\exists \beta,\beta' \in b: \\
\omega = \xi - \eta + \zeta , \quad 
\zeta = \eta +\beta, \quad   \xi = \eta + \beta' 
\end{array}} \, .
\end{align}
Given three sections $\sigma,\sigma',\sigma''$, 
we let  $\xi = \sigma(u)$, $\eta = \sigma'(u)$, $\zeta = \sigma''(u)$ so that
$\omega := (\sigma \sigma' \sigma'') (u) = \sigma(u) - \sigma'(u)+\sigma''(u) =
\xi - \eta + \zeta$, which is  the first condition in (\ref{NonbalancedEquation}).
The other two conditions just say that
$\xi$, $\eta$ and $\zeta$ belong to the same coset $\eta + b$.
Thus the ternary structures from (1), (2) and (3) agree; since (2) and (3) are torsors,
the semitorsor law from (1) actually defines a torsor structure on ${ }^\top b$.

\ssk
The same arguments apply  to sections of $\Omega \to b\backslash \Omega$, defining the
two torsor structures on $b^\top$.
\end{proof}

\section{An operator calculus on groups}\label{Sec:Operator}

In this section we generalize the various
``projection operators'' used in \cite{BeKi10} to the case of general groups. 
If $\Omega$ is abelian, then all operators are $\Z$-linear, or
$\Z$-affine, maps; in general, however, they will {\em not} be endomorphisms
of $\Omega$. In the following, all sums $F+G$ and differences 
$F-G$ of maps $F,G:\Omega \to \Omega$ are pointwise sums, resp.\ differences, and
hence one has to respect orders in such expressions.

\begin{definition}\label{ProjectionDefinition}
Assume $a$ and $x$ are  subsets of a group $\Omega$ such that $a \top x$.
The {\em left}, resp.\ {\em right projection operators} are defined by
$$
P^a_x:\Omega \to \Omega, \quad \omega = \alpha + \xi  \mapsto \xi 
$$
$$
\check P^x_a:\Omega \to \Omega, \quad \omega =  \alpha + \xi \mapsto \alpha 
$$
where $\alpha \in a$, $\xi \in x$.
\end{definition}

\begin{remark}
By dropping the transversality conditions, some results of this section still hold
if one replaces projectors by {\em generalized projectors}, i.e., by relations
$P^x_a \subset \Omega^2$ of the form
$\{ (\eta,\omega)\in \Omega^2 \mid \exists \alpha \in a, \xi \in x:
\eta = \alpha + \xi, \omega = \xi \}$, for arbitrary $x,a \in \cP$
(see \cite{BeKi10b} for the case of abelian $\Omega$.)
This will be taken up elsewhere.
\end{remark}

\begin{lemma}\label{ProjectionLemma}
Let $a,b,x,y \in \cP$ such that $a \top x,y$ and $a,b \top x$. Then
\begin{enumerate}[label=\roman*\emph{)},leftmargin=*]
\item
$ \check P^x_a + P^a_x = \id_\Omega$, that is, $\check P^x_a = \id_\Omega - P^a_x$ and
$P^a_x=-\check P^x_a + \id_\Omega$, 
\item
$P^a_x \circ P^b_x = P^b_x$; in particular,
$P^a_x$ is idempotent: $(P^a_x)^2 = P^a_x$,
\item
if $a$ is a subgroup, then $P^a_x \circ P^a_y = P^a_x$.
\end{enumerate}
\end{lemma}

\begin{proof}
i)
If $\omega = \alpha + \xi$ with $\alpha \in a$, $\xi \in x$, 
then $\xi = P^a_x(\omega)$ and $\alpha = \check P^x_a(\omega)$, whence the claim.
(Note: in general, $P^a_x + \check P^x_a$ will be different from the identity map!)

\ssk
ii) is obvious, and iii) is proved by decomposing
$\omega = \alpha + \xi = \alpha' + \eta$; then
$P^a_x \circ P^a_y (\omega) = P^a_x(\eta) = \xi = P^a_x(\omega)$
since $\eta = -\alpha' + \alpha + \xi$.
\end{proof}
 
\subsection{The ``unbalanced operators''}

\begin{definition}\label{TransvectionDefinition}
Let $a,b,x,y,z \in \cP$ such that  $x,y\top b$ and $a \top x,y$.
We define two types of {\em transvection operators (from $y$ to $x$, along $b$, reps.\ $a$)} by
\begin{align*}
T^b_{x,y} &:=  \id  - P^x_b + P^y_b  = \check P^b_x + P^y_b = \check P^b_x - \check P^b_y + \id 
\cr
\check T^a_{x,y} &:= \check P^y_a - \check P^x_a + \id   = \check P^y_a + P^a_x  = 
\id - P_y^a + P^a_x   \, .
\end{align*}
\end{definition}

\nin
Notation is chosen such that, for any operator $A:\Omega \to \Omega$ which can be written
as a ``word''  (iterated sum)
 in  projection operators, $\check A$ denotes the corresponding operator obtained
by replacing $\Omega$ by $\Omega^{opp}$.
 
\begin{lemma}\label{OperatorLemma1}
Assume $a,b$ are subgroups, $x,y \top b$ and $a \top x,y$. Then,
for all $z \in \cP$,
\begin{eqnarray*}
\Sigma(b,x,y,z) &=& T^b_{xy}(z) ,
\cr
\check \Sigma(a,x,y,z) &=& \check T^a_{xy}(z) .
\end{eqnarray*}
In particular, it follows that, if also $u \top b$,
$$
T^b_{x,x}=\id_\Omega, \qquad
T^b_{x,u}\circ T^b_{u,y}=T^b_{x,y}.
$$
Thus $\{ T^b_{x,y} \mid \, x,y \in { }^\top b \}$ is a group, isomorphic to
${ }^\top b$ as a torsor.
Similar remarks apply 
to $\check T^b_{x,y}$ with respect to  $b^\top$.
\end{lemma}

\begin{proof}
$(\check P^b_x - \check P^b_y + \id )(z)$ is the set of $\omega \in \Omega$ that can be
written $\omega = \xi - \eta +\zeta$ with $\zeta \in z$, $\xi = P^b_x(\zeta)$, $\eta = P^b_y(\zeta)$.
The last two conditions mean that there are $\beta,\beta' \in b$ with
$\zeta = \xi + \beta$ and $\zeta = \eta + \beta'$. 
Since $b$ is a subgroup, we see that $\omega$ satisfies precisely the
three conditions from (\ref{NonbalancedEquation}), whence the first claim. From 
the torsor property, we get now
$T^b_{x,u}(T^b_{u,y} (z)) = T^b_{x,y}(z)$.
Since this holds in particular for singletons $z=\{ \zeta \}$, the identity 
$T^b_{x,u}\circ T^b_{u,y}=T^b_{x,y}$ for operators on $\Omega$ follows.
The remaining claims are clear.
\end{proof}

\subsection{The ``balanced operators''} 
 
\begin{definition}\label{TranslationDefinition}
Let $a,b,x,y,z \in \cP$. If  
 $a \top x$ and $z \top b$, we define the {\em middle multiplication operators} by
\begin{eqnarray*}
M_{xabz}&: =&  P^a_x - \id + \check P^b_z  \cr
&=& P^a_x - P^z_b \cr
&=&  -\check P^x_a + \id - P^z_b  \cr
&=& - \check P^x_a + \check P^b_z  .
\end{eqnarray*}
If $a \top x$ and $y \top b$, we define the {\em left multiplication operators} by
\begin{eqnarray*}
L_{xayb} &:=& - \check P_a^x \circ \check P^b_y + \id ,
\end{eqnarray*}
and if $(-a) \top y$ and $z \top b$, we define the 
{\em right multiplication operators} by
\begin{eqnarray*}
R_{aybz}  &:=&  \id - P^z_b \circ P^{-a}_y \, .
\end{eqnarray*}
\end{definition}

\begin{lemma}\label{OperatorLemma}
Let $a,b,x,y,z \in \cP$.
\begin{enumerate}[label=\roman*\emph{)},leftmargin=*]
\item If 
 $a \top x$ and $z \top b$, then 
$\Gamma(x,a,y,b,z)  =  M_{xabz} (y)$.
\item If  $a \top x$ and $y \top b$, then 
$\Gamma(x,a,y,b,z)  =  L_{xayb} (z) $.
\item
If  $z \top b$ and $(-a) \top y$, then 
$\Gamma(x,a,y,b,z)  =  R_{aybz}(x) $.
\end{enumerate}
\end{lemma}

\begin{proof} 
i) 
$(P^a_x - \id + \check P^b_z) (y)$ is the set of all $\omega \in \Omega$ that can be written in the form
$\omega = \xi - \eta + \zeta$ with $\xi = P^\ga_x(\eta)$ and $\zeta = \check P^b_z(\eta)$ for some
$\eta \in y$. This means, in turn, that there is $\alpha \in a$ and $\zeta \in z$ such that
$\eta = \alpha + \xi$ and $\eta = \zeta + \beta$.
Summing up, $\omega$ satisfies exactly the three conditions from (\ref{2.3}), hence
they describe $\Gamma(x,a,y,b,z)$, whence the equality of both sets.

ii) 
$(-\check P^x_a \circ \check P^b_y + \id)(z)$ is the set of all $\omega \in \Omega$ that can
be written as
$\omega = - \alpha + \zeta$ with $ \alpha = P^x_a \circ \check P^b_y (\zeta)$ and 
$\zeta \in z$. This means that there is $\xi \in x$ such that
$\eta  = \alpha + \xi$ for
$\eta = \check P^b_y(\zeta)$.
And this means that there is $\eta \in y$ and $\beta \in b$ such that
$\zeta = \eta - \beta$. Thus we end up with the three conditions
$$
\omega = - \alpha +\zeta, \quad
\eta = \alpha + \xi,\quad
\zeta = \eta - \beta
$$
which are equivalent to the structure equations.

iii) 
$(\id - P^z_b \circ P^{-a}_y)(x)$ is the set of all $\omega \in \Omega$ that can be written 
as $\omega = \xi - \beta$ with $\xi \in x$ and $\beta := P^z_b (P^{-a}_y (\xi))$.
This means that there is $\beta \in b$ and $\zeta \in z$ such that
$\eta:=P^{-a}_y(\xi)= \zeta + \beta$. And this means that there is $\alpha \in a$ and
$\eta \in y$ such that $\xi = -\alpha + \eta$. Again, the three conditions thus obtained,
$$
\omega = \xi - \beta, \quad
\eta = \zeta + \beta,\quad
\xi = -\alpha + \eta
$$
are equivalent to the structure equations.
(Note: it is not assumed in this lemma that $a$ or $b$ are groups.)
\end{proof}

\begin{lemma}\label{IdempotentLemma}
Let $a,b,x,y,z \in \cP$. Then:
\begin{enumerate}[label=\roman*\emph{)},leftmargin=*]
\item
Assume $b$ contains the origin of $\Omega$. 
If $(-a) \top y$ and $y \top b$, then $R_{ayby} = \id$, hence $\Gamma(x,a,y,b,y)=x$ for all $x \in \cP$.
\item
If $x \top b$ and $a \top x$, then $L_{xaxb} = \id$, hence $ \Gamma(x,a,x,b,z)=z$ for all $z \in \cP$.
\end{enumerate}
\end{lemma}

\begin{proof} 
i) We have $P^y_b(y)=\{o\}$ since $o\in y$, hence
 $R_{ayby}= \id-P^y_b P^{-a}_y = \id $.

ii) $L_{xaxb} = - \check P^x_a \check P^b_x +\id =
- (\id - P^a_x) \circ \check P^b_x + \id =
- (\check P^b_x - \check P^b_x) + \id = \id$
\end{proof}

\subsection{The ``canonical kernel''}

\begin{definition}\label{KernelDefinition}
The {\em canonical kernel} is the family of maps defined by 
\begin{eqnarray*}
\Kappa_{x,y}^a &:= & P^a_x \vert_y:y \to x, \ \eta \mapsto P^a_x(\eta), \cr
\check \Kappa_{x,y}^b &:=& \check P^b_x \vert_y:y \to x, \ \eta \mapsto \check P^b_x(\eta), 
\end{eqnarray*}
where 
$x,y\in \cP$, $a,b \in \Gras(\Omega)$ with
$a\top x$ and $x \top b$. We let
$$
\Beta_{y}^{a,x,b} := K_{y,x}^a \circ \check K_{x,y}^b =
 P^a_y \circ \check P_x^b\vert_y: y \to y, \quad
\eta \mapsto P^a_y \circ \check P_x^b (\eta).
$$
\end{definition}

\begin{lemma}\label{KernelLemma}
Let $x,y\in \cP$, $a,b \in \Gras(\Omega)$ with $a\top x$ and $y \top b$. 
Then $\Kappa_{x,y}^a:y \to x$ is bijective iff $a \top y$, and if this holds, then,
for all $\eta \in y$,
$$
\Kappa^a_{x,y}(\eta) = \check T^a_{x,y}(\eta) = L_{xayb}(\eta) .
$$
Similarly, $\check \Kappa_{y,x}^b:x\to y$ is bijective iff $x \top b$,
and if this holds then, for $\xi \in x$,
$$
\check \Kappa^b_{y,x}(\xi)= T^b_{y,x}(\xi)=R_{axby}(\xi) .
$$
It follows that
 $a \top x$ if and only if $\Beta_y^{b,x,a}:y \to y$ is bijective,
and if this is the case, 
$$
\Beta_y^{b,x,a} = \check T^a_{y,x} \circ T^b_{x,y}\vert_y : y\to y , \qquad
\bigl( \Beta_y^{b,x,a} \bigr)\inv = T^b_{y,x} \circ \check T^a_{x,y}:y\to y \, .
$$
\end{lemma}

\begin{proof}
It is clear that $y$ is another set of representatives for $a\backslash \Omega$ iff
the projection from $y$ to $x$ is a bijection.
If this holds, then
$$
\check T^a_{x,y}(\eta)= (\check P^y_a + P^a_x)(\eta)=P^a_x(\eta),
$$
$$
L_{xayb}(\eta)= -\check P^x_a \check P^b_y (\eta) + \eta =
(-\check P^x_a + \id) \eta = P^a_x(\eta).
$$
The second claim is proved in the same way, and the last statement follows.
\end{proof}

If $x,y,a,b$ are vector lines in $\bR^2$, then $\Beta^{a,x,b}_y$ is the linear map obtained
by first  projecting $y$ along $b$ onto $x$ and than back onto $y$ along $a$.
Here is a quite general description that applies to the case of vector spaces,
where one sees the close relation with the so-called {\em Bergman operators}
known in Jordan theory:

\begin{lemma}\label{KernelLemma2}
Assume $\Omega$ is a direct product of its subgroups $y$ and $b$,
and let $a$ be a subgroup such that $y \top a$ and $x \in \cP$ such that $x \top b$.
Realize $x=G_X$ as a graph of a map $X:y \to b$ and $a$ as a graph of a
group homomorphism $A:b \to y$. Then 
$$
\Beta_y^{a,x,b} = \id_y - A \circ X: y \to y .
$$
In particular, $x \top a$ if, and only if, $\id_y -A \circ X$ is bijective.
\end{lemma}

\begin{proof}
We have to show that, for all $\eta \in y$, 
$P^a_y (\check P^b_x \eta) = \eta - AX\eta$.
Indeed:
since $\eta = \eta + X\eta - X\eta$ is a decomposition according to $x \top b$,
we have $ \check P^b_x \eta = \eta + X \eta$.
Next, we decompose
$\eta + X \eta =  \eta - AX \eta + AX \eta + X \eta$ with
$\eta -AX \eta \in y$ and $AX \eta + X \eta \in G_A =a$, wence
$P_y^a (\check P^x_b \eta) = P_y^a (\eta +X\eta) = \eta - AX\eta$.
\end{proof}

\section{The balanced  torsors $U_{ab}$}\label{Sec:The torsors}

\begin{theorem}\label{PairTheorem}
Fix a pair $(a,b)$ of subgroups of $\Omega$. Then the maps 
$$
\Pi_{ab}^+ : a^\top \times { }^\top b \times a^\top \to a^\top, \quad
(x,y,z) \mapsto \Gamma(x,a,y,b,z)
$$
$$
\Pi_{ab}^-:  { }^\top b \times a^\top \times { }^\top b \to  { }^\top b, \quad
(x,y,z) \mapsto \Gamma(x,a,y,b,z)
$$
are well-defined.
\end{theorem}

\begin{proof}
Assume $a\top x$ and use the bijection $\Omega \cong a \times x$ in order to write
other subsets as graphs: let $a\top z$ and write $z=\check G_Z$  with a map $Z:x\to a$, and let
 $y \top b$. 
We have to show that
$L_{xayb} (z)$ belongs again to $a^\top$, that is, that it can be written as  is a graph.
In order to prove this, let $\zeta = Z \xi + \xi \in z$, where $\xi \in x$.
Then
$$
L_{xayb}(Z\xi + \xi) =(  - \check P_a^x \circ \check P^b_y + \id)
(Z \xi + \xi) =
 - \check P_a^x \circ \check P^b_y (Z \xi + \xi) + Z \xi + \xi \ .
$$
Define the map
$$
F:x \to a, \quad
\xi \mapsto  - \check P_a^x \circ \check P^b_y (Z \xi + \xi)
$$
so that $L_{xayb}(\zeta)=F(\xi) + Z(\xi) + \xi$,
hence $L_{xayb}z$ is the graph of $F+Z:x \to a$, and it follows that $\Pi^+_{ab}$ is well-defined.

\ssk
Next assume that $x\top b$, $z \top b$ and $a \top y$ and identify $\Omega \cong z \times b$.
 Write $x = G_X$ as graph of a map $X: z \to b$.
We have to show that $R_{aybz}x$ belongs again to ${ }^\top b$. Let $\xi = \zeta + X \zeta \in x$, so
$$
R_{aybz}(\zeta + X \zeta) =
(\id -  P^z_b \circ P^a_y)(\zeta + X \zeta) =
\zeta + X \zeta -    P^z_b ( P^a_y (\zeta + X \zeta))
$$
and as above we see that $R_{aybz}x$ is the graph of a map $X + F:z \to b$, hence belongs
to ${ }^\top b$. Thus $\Pi^-_{ab}$ is well-defined.
\end{proof}

\begin{remark} See  \cite{BeKi10}, Theorem 1.8 for the case of abelian  $\Omega$ and linear
maps: in this case,  one can give 
 a  more explicit form for $F$ in terms of block matrices.
\end{remark}

\begin{theorem}\label{BalancedTorsorTheorem}
Let $\Omega$ be a group and $(a,b)$ a pair of  subgroups of $\Omega$. Then
\begin{enumerate}
\item
 the set
$
a^\top \cap { }^\top b = \{ x  \in \cP \mid \, a\top x, \, x \top b \}
$
is a subsemitorsor of $\cP_{ab}$, and with respect to the induced law it is a torsor, 
which we will denote by  $U_{ab}$. 
\item
 $b^\top \cap { }^\top a$ is a subsemitorsor of $\check \cP_{ab}$, and with the induced 
 law it becomes a
 torsor, denoted by 
$\check U_{ab}$,  
\item
the torsor $\check U_{ba}$ is the opposite torsor of $U_{ab}$:
$
\check U_{ba} = U_{ab}^{opp}$.
\end{enumerate}
\end{theorem}

\begin{proof} (1)
The fact that $a^\top \cap { }^\top b \subset \cP_{ab}$ 
is a subsemitotorsor follows directly from the preceding theorem. 
The idempotent laws are satisfied by Lemma \ref{IdempotentLemma}.
Thus $U_{ab}$ is a torsor.
Now (2) and (3) follow by the symmetry relation (\ref{Sym1}).
Note that the underlying sets of  $\check U_{ba}$ and $U_{ab}$ obviously agree.
\end{proof}

\begin{definition}
The {\em tautological bundle of $\cP(\Omega)$} is the set
$$
\hat \cP= \hat \cP(\Omega) :=
\bigl\{ (y,\eta) \mid\  y \in \cP(\Omega), \eta \in y \bigr\} ,
$$
and the map $\pi:\hat \cP \to\cP$, $(y,\eta) \mapsto y$ is called
the {\em canonical projection}.
\end{definition}

\begin{theorem}\label{KernelTheorem}
Let $\Omega$ be a group, $(a,b)$ a pair of subgroups of $\Omega$, and fix
$y \in U_{ab}$, considered as neutral element of the group $(U_{ab},y)$ defined
by the preceding theorem. Then there are natural left actions 
\begin{eqnarray*}
U_{ab} \times \cP \to \cP, & & (x,z) \mapsto x.z := \Gamma(x,a,y,b,z) ,
\cr
U_{ab} \times \hat \cP \to \hat \cP, & & (x,(z,\zeta)) \mapsto x.(z,\zeta) :=
\bigl( x.z, L_{xayb}(\zeta) \bigr),
\end{eqnarray*}
by ``bundle maps'', i.e., we have $\pi (x.(z,\zeta)) = x. (\pi(z,\zeta))$.
Over $a^\top$, this action can be trivialized: it is  given in terms of the canonical kernel by
(whith $\eta \in y$)
$$
x.(y,\eta) = (x.y, \Kappa^a_{x,y} (\eta)) .
$$
Similarly, we have natural right actions of $U_{ab}$ on
$\cP$ and on $\hat\cP$, which commute with the left actions.
In particular, $U_{ab}$ acts by conjugation on the fiber over the neutral
element, and this action is given by the explicit formula
$$
U_{ab} \times y \to y, \quad \eta \mapsto x  \eta x\inv =
L_{xayb} \circ R_{axby} (\eta) =\Beta^{a,x,b}_y (\eta).
$$
\end{theorem}

\begin{proof}
Concerning the left action, everything amounts to proving the following identity for operators
on $\Omega$, with $x,x' \in U_{ab}$:
\begin{equation}\label{LeftoperatorIdentity}
L_{xayb}\circ L_{x'ayb} = L_{\Gamma(x,a,y,b,x'),ayb} .
\end{equation}
Note that the operator on the right hand side is well-defined since
$\Gamma(x,a,y,b,x') \in U_{ab}$, by the preceding theorem.
Now, para-associativity (Theorem \ref{SemitorsorTheorem}, 
combined with Lemma \ref{OperatorLemma})
shows that, applied to any subset $z \subset \omega$, both operators give the same
result. Taking for $z$ singletons, it follows that the operators coincide. 
The proof for the right action is similar, and the fact that both actions
commute again amounts to an operator identity
\begin{equation}
L_{xayb} \circ R_{ax'by} = R_{ax'by} \circ L_{xayb},
\end{equation}
which is proved by the same arguments as (\ref{LeftoperatorIdentity}). 
For $x=x'$, we use the definition of the canonical kernel (Definition \ref{KernelDefinition}) 
and get the action by conjugation.
\end{proof}

\begin{theorem}
Assume $(a,b)$ is a pair of central subgroups. Then the set 
$\Gras_{ab} := \Gras(\Omega) \cap U_{ab}$ is a subtorsor of $U_{ab}$ which
 acts from the left and
from the right on the Grassmannian $\Gras(\Omega)$ and on
the {\em Grassmann tautological bundle}
$$
\widehat \Gras (\Omega):= \{ (x,\xi) \mid x \in \Gras(\Omega), \xi \in x \}.
$$
\end{theorem}

\begin{proof}
This follows by combining the preceding result with Theorem \ref{SubgroupTheorem}.
\end{proof}

\section{Distributive law and ``affine picture''}\label{Sec:Distributive}
 
The following fairly explicit description of the group law of $U_{ab}$ is the
analog of the ``affine picture'' from the abelian and linear case given in
Section 1 of \cite{BeKi10}: 
 
\begin{theorem}\label{PictureTheorem}
Let $(a,b)$ be a pair of subgroups of $\Omega$ and $x,y,z \in \cP$ such that $x,y,z \top b$ and
$a \top x,y$. Write $x$ and $z$ as left graphs with respect to the decomposition
 $\Omega \cong y \times b$, i.e., 
$x = G_X$, $z=G_Z$ with maps $X,Z:y \to b$. 
Then 
$$
\Gamma(G_X,a,y,b,G_Z) = G_{X + Z \circ \Beta_y^{a,x,b}} \ ,
$$
i.e.,
 $\Gamma(x,a,y,b,z)$ is the graph of the map 
$
X + Z \circ \Beta_y^{a,x,b}:y\to b$,
where $\Beta_y^{a,x,b}:y\to y$ is the canonical kernel
(Definition \ref{KernelDefinition}).
\end{theorem}

\nin This ``affine formula'' may be written, by identifying $x$ with $X$ and $z$ with $Z$,
\begin{equation}\label{AffineFormula}
X \cdot_{a,y,b} Z = X + Z \circ \Beta^{a,X,b}_y .
\end{equation}
Here $y=o^+$, $b=o^-$ are fixed ``basepoints'', and $a$ is the ``deformation parameter''.

\begin{proof}
Recall from Lemma  \ref{OperatorLemma} that 
$$
\Gamma(x,a,y,b,z) = (P^a_x - \id + \check P^b_z) (y)
 = \{ P^a_x(\eta) - \eta + \check P^b_z(\eta) \mid \, \eta \in y \} \, .
$$
Let $\eta \in y$. 
Since $P_x^a (\eta) \in x=G_X$ and $\check P^b_z(\eta) \in z=G_Z$, 
there exist unique $\eta' \in y$  and $\eta'' \in y$ such that
$$
P_x^a(\eta) = \eta' + X \eta' , \qquad \check P^b_z(\eta) = \eta'' + Z \eta'' \ .
$$
We determine $\eta'$ and $\eta''$ as functions of $\eta$: 
 since $\eta' \in y$ and $X \eta' \in b$, we have, by definition of the projection,
$$
\eta' = \check P^b_y (P_x^a (\eta)) = \Beta\inv (\eta),
$$
where $\Beta := \Beta_y^{a,x,b}:y \to y$ is the canonical kernel (Definition \ref{KernelDefinition}), and
in the same way, using Lemma \ref{ProjectionLemma}, we get
$$
\eta'' = \check P^b_y \check P^b_z (\eta) = \check P^b_y \eta = \eta ,
$$
whence $\check P_z^b(\eta)= \eta + Z \eta$. 
Since the operator $\Beta:y\to y$ is bijective (Lemma \ref{KernelLemma}), we can make
a change of variables $\eta' = \Beta\inv \eta$, $\eta=\Beta \eta' $, and we get
\begin{eqnarray*}
(P^a_x - \id + \check P^b_z) (\eta)  & = &
\eta' + X \eta' - \eta + \eta +  Z\eta
\cr 
&=&
\eta' + X \eta' + Z \Beta  \eta'
\cr
&=&
\eta' + (X + Z \circ \Beta) \eta' \, ,
\end{eqnarray*}
and hence, invoking Lemma \ref{OperatorLemma},  $\Gamma(x,a,y,b,z)$ is equal to
$$
(P^a_x - \id + \check P^b_z) (y)= \bigl\{ \eta' + (X + Z \circ \Beta_y^{a,x,b}) \eta' \mid \,
\eta' \in y \bigr\},
$$
that is, to  the (left) graph of the map
  $X + Z \circ \Beta_y^{a,x,b} : y \to b$.
\end{proof}

\begin{remark}
One may turn everything also the other way round: assume $(b,+)$ is a group,
$y$ a set, and let $G:=\Map(y,b)$. 
Assume given a map $\Beta :  G \to \Map(y,y)$, $X \mapsto \Beta^X$ such that $\Beta^0 = \id_y$,
and define a binary law on $G$ by
$$
X \cdot Z := X \cdot_B Z := X + Z \circ \Beta^X \ ,
$$
where $+$ denotes ``pointwise addition'' in $G$. 
It is straightforward to show that
this law is associative iff $\Beta$ becomes a homomorphism in the sense that
$$
\Beta^{X + Z \circ \Beta^X} = \Beta^X \circ \Beta^Z \ 
$$
(cf.\  \cite{Pi77}, p.\ 243, for a similar construction in the context of near-fields).
The neutral element is the zero map $0$, and
an element $X$ in $G$ is invertible iff $\Beta^X:y \to y$ is bijective, and then
its inverse is the ``quasi-inverse''
$$
X\inv:= - X \circ (\Beta^X)\inv .
$$
As a special case, all this works if $y$ and $b$ are  groups,
$A: b \to y$ a group homomorphism and $\Beta^X = \id_y+ A \circ X$, namely,
this is the affine picture of the following
\end{remark}

\begin{corollary}
Assume that $y$ is a subgroup commuting with $b$, and write $a=G_A$ with
a group homomorphism $A:b \to y$. Then Formula (\ref{AffineFormula})  reads
$$
\Gamma(G_X,G_A,y,b,G_Z) = G_{X + Z \circ (\id_y - A \circ X)} \ ,
$$
and $U_{ab} \to \Bij(y)^{op}$, $G_X \mapsto \id_y - AX$ is a group homomorphism.
\end{corollary}

\begin{proof} 
Write (\ref{AffineFormula}), using that by
 Lemma \ref{KernelLemma2},
$\Beta^{a,X,b}_y = \id_y - A \circ X$. 
The homomorphism property follows from
Theorem \ref{KernelTheorem}.\footnote{
We have to use the opposite group structure on $\Bij(y)$ in order to be in keeping
with our convention on $U_{ab}$, cf.\ footnote 4}
\end{proof}

\begin{theorem}\label{DistributiveTheorem}
Let $(a,b)$ be a pair of subgroups of $\Omega$. Then we have the following
``left distributive  law'': for all $x,y \in U_{ab}$ and $u,v,w \in U_b$,
$$
(xy (uvw)_b )_{ab} = \bigl( (xyu)_{ab} (xyv)_{ab} (xyw)_{ab} )_{b} .
$$
In other words, left multiplications $L_{xayb}$ from $U_{ab}$ are automorphisms of the torsor
$U_{b}$.
Similarly,  right multiplications  from $\check U_{ba}$ are automorphisms of
the torsor $U_a\check{ }$.
\end{theorem}

\begin{proof}
Let $u,v,w \top b$, and denote by uppercase letters the
corresponding maps $y\to b$.
Then the law of the ``pointwise torsor'' $U_b$ is simply described by the pointwise
torsor structure $U-V+W$ of maps from $y$ to $b$ (Theorem \ref{PointwisetorsorTheorem}).
The claim now follows from Theorem \ref{PictureTheorem}:
\begin{eqnarray*}
X + (U-V+W)  \circ \Beta^x &=&
X + U\circ \Beta^x - V \circ \Beta^x + W \circ \Beta^x \cr
&=&
(X + U \circ \Beta^x) - (X + V \circ \Beta^x) + (X + W \circ \Beta^x) \, .
\end{eqnarray*}
The ``dual'' statement follows by replacing $\Omega$ by $\Omega^{opp}$.
\end{proof}  
 
In general, the law of $U_{ab}$ is {\em not} right distributive: 
the laws $xz=\Gamma(x,a,y,b,z)$ and
$x+z=\Sigma(b,x,y,z)$ define a {\em near-ring}, and not a ring
(cf.\ Definition \ref{NearringDefinition}).   

\section{Permutation symmetries}\label{Sec:Symmetry} 

We have already mentioned (remark after Lemma \ref{SymmetryLemma}) that
the structure map $\Gamma$ and the structure space $\Gamm$ have certain
invariance, or ``covariance'',  properties with respect to permutations.
We  start by a simple remark on torsors.

\begin{definition}
The {\em torsor graph} of a torsor $(G,(\quad) )$ is
$$
\bT:= \bT(G):=\bigl\{ (\xi,\eta,\zeta,\omega) \in G^4 \mid  \, \omega = (\xi \eta  \zeta)   \bigr\} .
$$
Using an additive notation,  the {\em torsor graph} of a group $(\Omega,+)$ is thus given by
$$
\bT= \bT(\Omega)=\bigl\{ (\xi,\eta,\zeta,\omega) \in \Omega^4 \mid  \, \omega = \xi - \eta + \zeta  \bigr\} .
$$
\end{definition}

\begin{lemma}\label{TorsorLemma}
The torsor graph of a torsor
is invariant under the Klein four-group generated by the two double-transpositions
$(12)(34)$ and $(13)(24)$.
\end{lemma}

\begin{proof} In additive notation,
this follows immediately from the fact that the torsor equation 
$\omega = \xi - \eta + \zeta $
is equivalent to
$\eta = \zeta - \omega + \xi$ 
and to
$\zeta = \eta - \xi + \omega$.

For an intrinsic proof, without fixing a base point, note that
symmetry under $(13)(24)$ is equivalent so saying that the middle multiplication operators
$M_{xz}(y)=(xyz)$ are invertible with inverse $M_{zx}$, and 
symmetry under $(12)(34)$ is equivalent so saying that the left multiplication operators
$L_{xy}(z)=(xyz)$ are invertible with inverse $L_{yx}$ (cf.\, Appendix A of \cite{BeKi10}).
\end{proof}

Recall the definition of the structure space $\Gamm(\Omega)$ 
(Definition \ref{StructurespaceDefinition}) 
and the equivalent
versions of the structure equations (Lemma \ref{SymmetryLemma}).
Note that in System (\ref{EquationS3}) the ``torsor equation'' appears, hence
 the  ``$(\xi,\eta,\zeta,\omega)$-projection''
$$
\Gamm (\Omega) \to \bT(\Omega) , \quad
\bigl(\xi,\zeta;\alpha,\beta;\eta,\omega\bigr) \mapsto (\xi,\eta,\zeta,\omega)
$$
is well-defined.  Concerning other variables, the ``torsor equation'' also appears, modulo certain
sign changes. The relevant symmetry group here is 
 a subgroup $\bV$ of the permutation group $S_6$ playing a similar role as
the Klein four-group $V \subset S_4$ in the preceding lemma:

\begin{definition}\label{BigKleinDefinition}
The {\em Big Klein group} is the subgroup $\bV$ of permutations  $\sigma \in A_6$, 
 acting on six letters $\{ \alpha,\beta,\xi,\zeta,\eta,\omega\}$, and preserving the partition
$$
\{ \alpha,\beta,\xi,\zeta,\eta,\omega \} = 
A_1 \cup A_2 \cup A_3, \qquad
A_1:= \{ \alpha,\beta \}, \
A_2:= \{ \xi,\zeta \}, \
A_3 := \{ \eta,\omega \} ,
$$
i.e., for $i=1,2,3$, there is $i' \in \{ 1,2,3 \}$ with
$\sigma (A_i) = A_{i'}$.
\end{definition}

\begin{lemma}\label{BigKleinLemma}
The  Big Klein group $\bV$ is  isomorphic to $S_4$, and  its  action on six letters
 is equivalent to the natural action of $S_4$ on the set $K$ of all two-element subsets 
of $\{1,2,3,4\}$.
\end{lemma}

\begin{proof}
We fix the following correspondence between our six letters and $K$:
$$
\alpha = \{ 1,2 \}, \
\beta = \{ 3,4 \}, \
\xi = \{ 1,3 \}, \
\zeta = \{ 2,4 \}, \
\eta = \{ 1,4 \}, \
\omega = \{ 2,3 \} .
$$
The natural action of $S_4$ induces a homomorphism
$S_4 \to S_6$, letting act $S_4$ on the six letters $\alpha,\ldots,\omega$.
This homomorphism is obviously injective, and its image belongs to $\bV$
(note that each transposition from $S_4$ acts
by a double-transposition of these six letters, hence the image belongs to $A_6$).
Let us prove that the homomorphism is surjective:
from the very definition of $\bV$ we get a homomorphism
$\bV \to S_3$, sending $\sigma$ to the permutation $i \mapsto i'$.
The kernel of this homomorphism is a Klein four-group, and
one easily constructs a section $S_3 \to \bV$, so that 
$\vert \bV \vert = 24 = \vert S_4 \vert$, whence the claim. 
\end{proof}

\begin{definition} 
A vector $s=(s_1,\ldots,s_6)$ with $s_i \in \{ \pm 1 \}$ will be called a {\em sign vector}.
Given a sign vector $s$,  the subspace
$$
\Gamm^{s} :=
\Bigsetof{ (\xi,\zeta; \alpha,\beta; \eta,\omega) \in \Omega^6}
{ \bigl(s_1 \xi,s_2 \zeta; s_3 \alpha,s_4 \beta; s_5 \eta, s_6 \omega \bigr) \in \Gamm   } 
$$
is called a {\em signed version of the structure space}. 
\end{definition}

\nin 
Since $\Omega \to \Omega$, $v \mapsto -v$ is an antiautomorphism of $\Omega$, we see that the 
opposite structure space is a signed version of $\Gamm$, namely
$$
\check \Gamm = \Gamm^{(-1,-1;-1,-1;-1,-1)}  \, .
$$

\begin{theorem}\label{BigKleinTheorem}
The Big Klein group transforms $\Gamm$ into signed versions of $\Gamm$:
for each $\sigma \in \bV$ there exists a sign vector $s=s(\sigma)$
such that
$\sigma . \Gamm = \Gamm^{s(\sigma)}$.
More precisely, we have the following table of elements
$\sigma \in \bV$ (given together with their corresponding element in $S_4$, under the
isomorphism from Lemma \ref{BigKleinLemma}) and corresponding
sign vectors $s(\sigma)$:

\msk
\nin
\begin{tabular}{l | l | l}
element of $V\subset S_4$ & corresponding element $\sigma \in \bV$ & 
corresponding sign vector  $s(\sigma)$ \cr
\hline
$\id$ & $\id$ & $(1,1;1,1;1,1)$
\cr 
$(12)(34)$
&$(\xi \zeta)(\eta \omega)$
& $(1,1;-1,-1;1,1)$
\cr
$(13)(24)$ &
$(\alpha \beta)(\eta \omega)$&
$(-1,-1;1,1;-1,-1)$
\cr
$(14)(23)$ &
$(\alpha \beta)(\xi\zeta)$ &
$(-1,-1;-1,-1;-1,-1)$
\end{tabular}

\msk
\nin
\begin{tabular}{l | l | l}
element of $A_4 \setminus V$ & corresponding element $\sigma \in \bV$ & 
corresponding sign vector $s(\sigma)$ \cr
\hline
$(123)$
&$(\alpha \omega \xi)(\beta\eta \zeta)$
& $(1,-1;-1,1;-1,-1)$
\cr$(132)$
&$(\alpha \xi \omega)(\beta\zeta \eta)$
& $(-1,1;-1,-1;-1,1)$
\cr
$(124)$ &
$(\alpha \zeta \eta)(\beta \xi \omega)$ &
$(-1,1;1,1;1,-1)$
\cr
$(142)$ &
$(\alpha \eta \zeta)(\beta \omega \xi)$ &
$(-1,1;1,-1;1,1)$
\cr
$(134)$ & $(\alpha \omega \zeta)(\beta  \eta \xi)$ &
$(1,-1;-1,1;1;1)$
\cr
$(143)$ & $(\alpha  \zeta \omega)(\beta \xi \eta)$ &
$(1,-1;1,1;1,-1)$
\cr
$(234)$ & $(\alpha \xi \eta)(\beta \zeta \omega)$ &
$(1,-1;-1,-1;-1,1)$
\cr
$(243)$ & $(\alpha \eta \xi)(\beta \omega \zeta)$ &
$(-1,1;1,-1;-1,-1)$
\end{tabular}

\msk
\nin
\begin{tabular}{l | l | l}
transposition in $S_4$ & corresponding element $\sigma \in \bV$ & 
corresponding sign vector $s(\sigma)$ \cr
\hline
$(12)$ & $(\xi \omega)(\zeta \eta)$ & $(1,1;1,-1;1,1)$ 
\cr
$(13)$ & $(\alpha \omega)(\beta \eta)$ & $(1,-1;1,-1;1,-1)$ 
\cr
$(14)$ & $(\alpha \zeta)(\beta \xi)$ & $(1,1;1,1;1,-1)$
\cr
$(23)$ & $(\alpha \xi)(\beta\zeta)$ & $(-1,-1;-1,-1;-1,1)$
\cr
$(24)$ & $(\alpha \eta)(\beta\omega)$& $(-1,1;1,-1;1,-1)$
\cr
$(34)$ & $(\xi \eta)(\zeta \omega)$& $(1,1;-1,1;1,1)$
\end{tabular}

\msk 
\nin
\begin{tabular}{l | l | l}
elt.\  of order $4$ in $S_4$ & corresponding element $\sigma \in \bV$ & 
corresponding sign vector $s(\sigma)$ \cr
\hline
$(1234)$ & $(\alpha \omega \beta \eta)(\xi \zeta)$&
$( 1,-1;-1,1;-1,1)$
\cr
$(1243)$ & $(\alpha \zeta \beta \xi) (\eta \omega)$&
$( -1,-1;1,1;1,-1)$
\cr
$(1324)$& $(\alpha \beta) (\xi \omega \zeta \eta)$&
$(-1,-1;1,-1;-1,-1)$
\cr
$(1342)$& $(\alpha \xi \beta \zeta)(\eta \omega)$&
$( 1,1;-1,-1;-1,1)$
\cr
$(1423)$
& $(\alpha \beta)(\xi \eta \zeta \omega)$& 
$(-1,-1;-1,1;-1,-1)$
\cr
$(1432)$
& $(\alpha \eta \beta \omega)(\xi \zeta)$& 
$(-1,1;1,-1;-1,1)$
\end{tabular}
\end{theorem}

\begin{proof}
The proof is by direct computation:
take any of the systems from (\ref{EquationS2}) -- (\ref{EquationS7}) in 
Lemma \ref{SymmetryLemma},
replace variables according to $\sigma$;
the system thus obtained agrees, up to sign changes, with some other among
the systems from (\ref{EquationS2}) -- (\ref{EquationS7}), 
and, by comparing, one can read off the sign vector.
\end{proof}

A glance at the tables shows that elements from $A_4$ induce an even number
of sign changes, and elements of $S_4 \setminus A_4$ an odd number of sign
changes. The sign vectors given in the tables form an $S_4$-torsor under the induced action 
of $\bV$. If $\Omega$ is non-commutative, then different sign vectors $s$ give
rise to different spaces $\Gamm^s$; if $\Omega$ is commutative, then
$\Gamm = \check \Gamm$, but (except for some degenerate examples) this is the
only case in which two signed versions of the structure space coincide.
If, in the abelian case, we work only with {\em linear} subspaces (subgroups),
then we may ignore all sign changes, and hence all 24 permutations can
be considered as equivalent. 
In the general case,
the following statements also follow by direct inspection from the tables.

\begin{definition}
A subset $x$ in a group $(\Omega,+)$ is a
{\em symmetric subset} if $-x=x$.
\end{definition}

\begin{theorem}\label{SecondSymmetryTheorem}
For any group $(\Omega,+)$ and $(x,a,y,b,z) \in \cP^5$, we have:
\begin{enumerate}
\item (behavior of $\Gamma$ under the Klein group $\{ \id, (14),(14)(23),(23) \} \subset S_4$):

$\Gamma(b,z,y,x,a)=-\Gamma(x,a,y,b,z)$ 

$\Gamma(z,b,y,a,x)=\check \Gamma(x,a,y,b,z)$

$\Gamma(a,x,y,z,b) = - \check \Gamma(x,a,y,b,z)$,

and if $x,a,y,b,z$ are symmetric subsets, then
$ \Gamma(x,a,y,b,z)=\Gamma(a,x,y,z,b)$.
\item
If we consider sign changes in $\alpha$ or $\beta$ as  negligible (and the
the corresponding signed versions as ``equivalent''), then the equivalence class
of $\Gamm$ is invariant under  the Klein group $\{ \id, (12)(34),(12)(34)\} \subset S_4$.
\end{enumerate}
\end{theorem}

\begin{remark}
Following   \cite{BeKi10}, Theorem 2.11 and Remark 2.12, the second item may be
reformulated in another way:
invariance under $(12)$, that is, under $(\xi \omega)(\zeta \eta)$, amounts to 
the fact that the inverse of the relation ${\bf l}_{xayb} \subset \Omega^2$ 
(which, as in \cite{BeKi10}, generalizes the operator $L_{xayb}$ for arbitrary
$(x,a,y,b) \in \cP^4$) is given by the relation ${\bf l}_{yaxb}$.
Similarly, invariance under $(34)$ amounts to the analog  for right translations,
and invariance under $(12)(34)$ to the analog for middle multiplication operators:
$$
({\bf l}_{xayb})\inv = {\bf l}_{yaxb} , \qquad
({\bf r}_{aybz})\inv = {\bf r}_{azby}, \qquad
({\bf m}_{xabz})\inv = {\bf m}_{zabx} \, .
$$
Note that this is the exact analog of Lemma \ref{TorsorLemma}; however, since $a$, $b$
need not be transversal subgroups, these
relations now apply to {\em semitorsors} as well. In a certain sense, this means that
our semitorsors have the same ``symmetry type as a torsor'' --
a property which certainly distinguishes them from ``arbitrary'' semitorsors.
\end{remark}

\section{Generalized lattice structures}\label{Sec:Lattice}

For abelian groups, Theorem 2.4 of \cite{BeKi10} establishes a close link between the 
structure map $\Gamma$
and the {\em lattice of subgroups}. For non-abelian groups $\Omega$, 
the subgroups form no longer a lattice. Nevertheless, the two set-theoretic operations
$$
x \land y :=x \cap y, \qquad
x  +  y := \{ \omega \in \Omega \mid \, \exists \xi \in x, \,  \exists \eta \in y : \,
\omega = \xi + \eta \} 
$$
behave very much like ``meet'' ($\land)$ and ``join'' ($\lor$), as shows the following analog
of Theorem 2.4 of \cite{BeKi10}:

\begin{theorem} \label{LatticeTheorem}
Let $x,a,y,b,z$ be subgroups of $\Omega$. Then we have:
\begin{enumerate}[label=\emph{(}\arabic*\emph{)},leftmargin=*]
\item values of $\Gamma$ on the ``diagonal $x=y$'': 
\begin{eqnarray*}
\Gamma(x,a,x,b,z) &=&
\Bigl( \bigl(x \land a\bigr) +  z\Bigr) \land (x +  b) \cr
&=&
(x \land a) +  \Bigl(z \land (x +  b) \Bigr),
\end{eqnarray*}
\item
values of $\Gamma$ on the ``diagonal $a=z$'': 
\begin{eqnarray*}
\Gamma(x,a,y,b,a)&=& \Bigl( \bigl(x \land(a +  y) \bigr)+  b\Bigr) \land a \cr
&=& \Bigl(x +  \bigl( (y +  a)\land b\bigr) \Bigr) \land a \ ,
\end{eqnarray*}
\item
values of $\Gamma$ on the ``diagonal $b=z$'': 
\begin{eqnarray*}
\Gamma(x,a,y,b,b)&=& \Bigr( \bigl( a +   (y \land b) \bigr) \land x \Bigr) +  b \cr
&=&
\Bigl(a \land \bigl( x+  (y \land b) \bigr) \Bigr) +  b \ .
\end{eqnarray*}
\end{enumerate}
This implies, in particular, that for all $x,a,y,b \in \Gras(\Omega)$,
\begin{eqnarray*}
\Gamma(x,a,x,b,x) &=& x \cr
\Gamma(a,a,y,b,b) &=&a +  b \cr
\Gamma(b,a,y,b,a) &=&a \land b \ .
\end{eqnarray*}
\end{theorem}

\begin{proof}
The technical details of the the proofs
are exactly  as in \cite{BeKi10}, loc.\ cit., 
by respecting the possible non-commutativity of $\Omega$ (and hence of the operation
$+ $), so let us here
 only prove the first equality from item (1).  As in loc.\ cit.\, 
 it is always understood that $\alpha \in a$, $\xi \in x$, $\beta \in b$,
$\eta \in y$, $\zeta \in z$.
We use System (\ref{EquationS3}).

\ssk
Let $\omega \in \Gamma(x,a,x,b,z)$, then
$\omega = \xi - \beta$, hence $\omega \in (x +  b)$,
and
$\omega = \xi - \eta  + \zeta$ with
$v:=\omega - \zeta =  \xi - \eta  \in x$ (since $x=y$).
On the other hand,
$v=\omega - \zeta = -\alpha \in a$, whence
$\omega = v + \zeta$ with $v \in (x \land a)$, proving one inclusion.
(Note that we have used that $x$ and $a$ are subgroups and that $b$ is symmetric.)

\ssk
Conversely, let $\omega \in ( (x \land a) +  z) \land (x +  b)$.
Then
$\omega =\xi +  \beta  = \alpha + \zeta$
with $\alpha \in (x \land a)$.
Let $\eta := - \alpha + \xi$. Then $\eta \in x$, and
$\omega = \xi + \beta =  \alpha + \eta +  \beta$, hence
$\omega \in\Gamma(x,a,x,b,z)$.

The remaining proofs are similar. Note that,
since the systems in (\ref{EquationS2})--(\ref{EquationS8}) come in pairs,
there are in fact two different expressions for one diagonal value. 
For the final conclusion, one uses the ``absorption laws'' in the following form:

\begin{lemma} Recall that $\cP^o = \{ x \in \cP(\Omega) \mid o \in x \}$.
\begin{enumerate}
\item
If $y \in \cP^o$, then  $x\land (x +  y)=x$.
\item
If $y \in \cP^o$ and $x \in \Gras(\Omega)$,  then 
  $x +  (x \land y)=x=(x \land y) +  x$.
\end{enumerate}
 \end{lemma}
\end{proof}

\begin{remark}
It is remarkable that for subgroups $x,a,z$ a ``non-commutative modular law'' is
still satisfied (cf. Remark 2.5 of \cite{BeKi10}): by letting  $b=x$  in (1), we get
$$
((x \land a) +  z) \land x   = 
(x \land a) +  (z \land x) \ ,
$$
and letting $a=x$ we get the ``dual modular law''
$$
x +  (z \land (x +  b)) =(x+  z) \land (x +  b) \, .
$$
More generally, one has the impression that the formulas express some sort of ``duality'' between
the operations $\land$ and $+ $. This would be rather mysterious since, at a first glance,
the definitions of $\land$ and $+ $ look quite ``non-symmetric''. 
\end{remark}

\section{Final remarks} 

Since {\em groups} and {\em projective spaces} are 
foundational concepts in mathematics, the present approach is likely to interact
with  many areas of mathematics. The author's original motivation came from 
{\em non-associative algebra}, and in particular {\em Jordan theory} --  this domain
does not really belong to today's mathematical mainstream, and thus
the shape of the approach
presented here may seem quite unusual for a ``normal user of group theory''.
The reader may find more motivation and general remarks in the introductory and
concluding sections of preceding work, e.g., \cite{BeKi10, BeKi10b, BeNe05, BeL08, 
Be02}.

\ssk
In the following, let us add some short comments on aspects for which
the generalization of the framework to general groups and
general torsors of mappings may be relevant. 
Most importantly, one needs to study {\em categorial aspects} more
systematically; we hope to come back to such items in subsequent work.

\subsection{Morphisms}
Projective spaces or general Grassmannians can be turned into categories
in two essentially different ways (cf.\ \cite{Be02, BeKi10}).
The same is true in the present context. 
One may conjecture that
an analog of the ``fundamental theorem of projective geometry'' holds:
{\em every automorphism of $(\cP(\Omega),\Gamma)$ is induced by an automorphism
of $\Omega$} (if $\Omega$ is not too small). 
However, the main difference is that now, in the non-abelian case,
 ``interior maps'' (left-, right-
and middle multiplication operators) are no longer morphisms in either of the
two categories. One may wonder whether they are morphisms in yet another sense.

\subsection{Antiautomorphisms, involutions}
For the case of commutative $\Omega$, see \cite{BeKi10b}.
If $\Omega$ is non-commutative, the situation seems to change drastically:
first of all, 
 in principle, each permutation of the $5$ arguments of $\Gamma$
gives rise to its own notion of ``anti-homomorphism''. 
The simplest case corresponds to anti-homomorphisms on the level of $\Omega$:
they correspond to ``usual'' morphisms between $\Gamma$ and $\check \Gamma$.
For instance, the inversion map of $\Omega$ induces an ``involution'' of this kind.
It corresponds  to the permutation $(\xi\zeta)(\alpha \beta)$, which belongs to $\bV$.
On the other hand, looking at permutations not belonging to $\bV$,
one may ask whether anti-homomorphisms corresponding to the
permutation $(\xi \zeta)$ exist: these would generalize the involutions
considered in \cite{BeKi10b}.
In the commutative case, they are given by {\em orthocomplementation maps}.
In particular, they induce {\em lattice antiautomorphisms}.
As mentioned above, our formulas suggest that some kind of duality in this sense
exists; on the other hand, there seems to be no hope to generalize orthocomplementation
maps in some obvious way. 

\subsection{Subobjects}
A {\em subspace of $\cP(\Omega)$} is a
subset $\cY \subset \cP(\Omega)$  stable under $\Gamma$.
Such sets may be defined by algebraic conditions (cf.\ Theorem \ref{SubgroupTheorem}),
or by topological or differential conditions: e.g.,  if $\Omega$ is a topological or
Lie group (and  in particular, for $\Omega = \bR^{2n}$), we may
 consider spaces $\cY$ of {\em closed} subsets or of {\em smooth} or {\em
algebraic submanifolds}.  One expects that $\Gamma$ will have ``singularities'', so
one possibly has to exclude some ``singular sets'': outside such sets 
we expect $\Gamma$ to be fairly regular
(cf.\ related results in \cite{BeNe05}). 
The analogy with near-rings suggests also to look at subspaces corresponding to
{\em  near-fields}.

\subsection{Ideals, inner ideals, intrinsic subspaces}
As in rings or near-rings, or in Jordan algebraic structures, one may define notions
of certain subobjects playing the r\^ole of various kinds of ``ideals'', and which will
be of  importance for a systematic ``structure theory''.
Such ideals are defined as subobjects $\cY$ defining conditions like
the ``inner ideal condition''
$\Gamma(\cY,\cP,\cY,\cP,\cY) \subset \cY$.
In a Jordan theoretic context, such sets have been characterized in \cite{BeL08} as
``intrinsic subspaces''.

\subsection{Products}
Our construction is compatible with {\em direct products}.
For instance, $\cP(\Omega_1) \times \cP_2(\Omega_2)$ is a subspace 
of $\cP(\Omega_1 \times \Omega_2)$. 
Here, the case $\Omega_1 = \Omega_2 = \Omega$ is particularly important
since elements of $\cP(\Omega\times \Omega)$ are nothing but
{\em relations} on $\Omega$, and hence may play the r\^ole of
{\em endomorphisms}, as discussed above.
This situation is characterized by the existence of 
{\em transversal triples} (Theorem \ref{TripleTheorem}) 
and of projections that are endomorphisms.

\subsection{Axiomatic approach, base points, and equivalence of categories}
Following \cite{BeKi10}, one may consider as {\em base point} a fixed pair
$(o^+,o^-)$ of transversal subgroups, and then look at the ``pair of ternary
near-rings'' 
$(\pi^+,\pi^-)$ defined by Theorem \ref{PairTheorem} as a sort of
``tangent object''. 
In which sense can one say, then, that the theory of such
objects is equivalent to the theory of $\cP(\Omega)$ --
is there an equivalence of categories between ``tangent objects''
and ``geometries with base point''?

\subsection{Symmetric spaces, and Jordan theory}
Unlike the three ``diagonals'' mentioned in Theorem \ref{LatticeTheorem},
the ``diagonal $x=z$'' is not related to lattice theory, but rather to
{\em symmetric spaces}: in any torsor, the ternary composition gives rise to
a binary map $\mu(x,y):=(xyx)$, which (in the case of a Lie group) is precisely
the underlying ``symmetric space structure'' (in the sense of \cite{Lo69}).
Thus, automatically, our torsors $U_{ab}$ give rise to families of symmetric spaces.
Other symmetric spaces can be constructed from these in presence of an {\em involution}.
Therefore, symmetric spaces are a main geometric ingredient for the theory
corresponding to the ``diagonal $x=y$'', which, in turn, would be a kind of 
``non-commutative Jordan theory''.  
Many of the surprising features of Jordan theory are due to the fact that,
because of the symmetry $(\alpha \zeta)(\xi \beta)$ (Theorem 
\ref{SecondSymmetryTheorem}), this theory essentially is the same as the
theory of the ``diagonal $a=b$'', i.e., of the family of torsors
$U_{aa}$: thus there should be some kind of duality between
certain families of symmetric spaces and certain families of torsors.

\subsection{Flag geometries}
For Jordan theory, the {\em projective geometry of a Lie algebra} defined in 
\cite{BeNe04} gives a  useful ``universal geometric model''.
The construction is similar in spirit to the present work;  however, it is not
clear at all how to carry it out on the level of groups.
Essentially, one needs a definition of an analog of our maps $\Gamma$ and
$\Sigma$ for {\em (short) flags}, that is, for pairs of subsets
$(x,x')$ with $o \in x \subset x' \subset \Omega$, instead of single subsets.
Some results by J.\ Chenal (\cite{Chenal}) point in the direction that such
constructions should be possible for general spaces of finite flags. 

\subsection{Reductive groups, finite groups}
For a reductive Lie group $\Omega$, how is the projective geometry of $\Omega$
related to well-known structure theory? Is there some link with the notion of
{\em building}?  Is the projective geometry of the {\em Weyl group} related in some 
definite way to the projective geometry of $\Omega$ ?

\subsection{Supersymmetry}
We have the impression that Section \ref{Sec:Symmetry} on ``symmetry'' is closely
related to the topic of {\em supersymmetry}: 
indeed, the behavior of the ``signed versions'' under permutations reminds the
``sign rule'' of supersymmetry. 
On a very fundamental level, our map $\Gamma$ takes account of the
principle that there is no reason to prefer a group to its opposite group: in
principle, both should play symmetric r\^oles. 
However, in presence of additional structure (two or more operations)
this symmetry may be broken: e.g., there may be ``left-distributivity'', but not
``right-distributivity''. As the tables in Theorem \ref{BigKleinTheorem} show,
a complete book-keeping of such situations is not entirely trivial.
The next level in such a book-keeping should be reached when we are
dealing with $\cP(\Omega \times \Omega)$: namely, subsets of $\Omega \times \Omega$
are just relations on $\Omega$, and here again, there is no reason
to prefer ``usual'' composition to its opposite; in other words, we have to
fix choices and conventions for the structure maps
$\Gamma_{\Omega \times \Omega}$ of the group $\Omega \times \Omega$,
and similarly for iterated products $\Omega^k$.
{\em Supersymmetry} might turn out to be one of the sign-rules that are used for
such book-keeping of iterated products.

\end{document}